\documentclass[11pt]{article}

\usepackage[T1]{fontenc}
\usepackage{amsmath,amsthm,mathtools}
\usepackage{libertine}
\usepackage[libertine]{newtxmath}
\usepackage{microtype}
\usepackage[letterpaper,margin=1in]{geometry}
\usepackage[colorlinks=true,linkcolor=blue,citecolor=blue,urlcolor=blue]{hyperref}

\usepackage{enumerate}

\usepackage{booktabs}
\usepackage{array}

\usepackage{float}

\usepackage[numbers]{natbib}
\theoremstyle{plain}
\newtheorem{theorem}{Theorem}[section]
\newtheorem{lemma}[theorem]{Lemma}
\newtheorem{proposition}[theorem]{Proposition}
\newtheorem{corollary}[theorem]{Corollary}

\theoremstyle{definition}
\newtheorem{definition}[theorem]{Definition}
\newtheorem{example}[theorem]{Example}

\theoremstyle{remark}
\newtheorem{remark}[theorem]{Remark}

\DeclareMathOperator{\Var}{Var}

\DeclareMathOperator{\Exp}{Exp}
\DeclareMathOperator{\Geom}{Geom}
\DeclareMathOperator{\CV}{CV}
\newcommand{\dK}{d_{\mathrm{K}}}

\title{On Completion Times under Memoryless Catastrophe}

\author{%
  Sichen Wang\\
  \small Shenzhen MSU-BIT University\\
  \small\texttt{wsc@smbu.edu.cn}
  \and
  Zhipeng Lu\\
  \small Shenzhen MSU-BIT University\\
  \small\texttt{zhipeng.lu@smbu.edu.cn}
}
\date{}

\hypersetup{%
  pdftitle={On Completion Times under Memoryless Catastrophe},
  pdfauthor={Sichen Wang, Zhipeng Lu}}

\begin{document}
\maketitle

\begin{abstract}
We study the completion time of a task subject to independent reset (catastrophe) at each step. The completion-time PGF depends on the base-process PGF through an affine relation, and we exploit this structure systematically. Our main result shows that, among age-based catastrophe mechanisms, geometric-tail catastrophe is exactly the class that yields uniform affine PGF structure; in continuous time, the characterization sharpens to Poisson resetting. We establish a sharp two-sided Kolmogorov bound of order $p+|\alpha|$ for the exponential approximation $\dK(T/E[T],\Exp(1))$, thereby closing a logarithmic gap. Applications to the coupon collector with reset coupons reveal a discontinuous Gumbel-to-Exponential transition under resetting, while a multi-phase model exhibits a Gaussian-to-exponential transition with exponential convergence rate.

\end{abstract}

\section{Introduction}\label{sec:intro}

\subsection{Motivation}

Stochastic processes subject to catastrophe undergo the total loss of accumulated progress at random intervals, necessitating a complete restart from the initial state. This abstraction of a base task disrupted by total resets arises across several disciplines:
\begin{itemize}
\item Randomized algorithms restarted upon timeout, where~\cite{LSZ93} showed that a universal restart strategy is near-optimal up to a logarithmic factor;
\item \emph{Queueing systems} flushed by disasters, initiated by the disaster models of~\cite{BGR82} and developed through the negative-customer framework of~\cite{Gelenbe91};
\item \emph{Diffusive search} under stochastic resetting, where~\cite{EM11} demonstrated that Poisson resetting of a Brownian particle produces a nonequilibrium steady state with finite mean first-passage time---launching a now-extensive literature surveyed in~\cite{EMS20}.
\end{itemize}

For such a process that must complete a task, its completion time, denoted by $T$, depends on two ingredients: the base task (how long without catastrophe?) and the catastrophe mechanism (when does reset strike?).  Much of the existing literature studies the first ingredient---optimizing restart timing for a given task~\cite{Reuveni16, PalReuveni17, CS18}.  We ask a complementary question about the second:

\begin{quote}
\emph{What algebraic structure does the catastrophe mechanism impose on the probability generating function of $T$, and what are the consequences for moments, limit laws, and convergence rates?}
\end{quote}

The question has a clear theoretical lineage. The probability generating function (PGF) formula for completion under geometric restart was derived independently by~\cite{BP21} and~\cite{FP21}.  Both observed the striking fact that $E[T]$ depends on the base process only through a single PGF evaluation---Flynn and Pilyugin called this ``wonderful.''  This is, in retrospect, the $k=1$ shadow of a deeper phenomenon.  Meanwhile, in branching process theory,~\cite{Agresti74} and~\cite{Sagitov13} exploited the fact that geometric compounding inherently produces linear fractional (M\"obius) PGFs; in the catastrophe/restart context, this connection has not been made. Our work identifies the algebraic mechanism underlying these observations---\emph{affine PGF structure}---and develops it into a systematic theory: a derivative reduction principle, a characterization theorem, an exponential limit with explicit Laplace expansion, sharp two-sided convergence rates, and a continuous-time counterpart that sharpens the characterization to full Poisson resetting. 


\subsection{Setup and the PGF Formula}

Throughout, $T_0$ is a positive-integer-valued, almost surely finite random variable (the base completion time). Catastrophe occurs independently at each step with probability $q \in (0,1)$.  
For an integer-valued random variable $X$, write $G_X$ for its probability generating function.  
We set
\[
  s = 1-q, \qquad g = G_{T_0}, \qquad p = g(s), \qquad D_k = g^{(k)}(s).
\]
We write $\Exp(\lambda)$ for the exponential distribution with rate $\lambda > 0$ (density $\lambda e^{-\lambda x}$, $x \geq 0$; mean $1/\lambda$).  In particular, $\Exp(1)$ denotes the standard exponential (mean $1$). The attempt decomposition
\begin{equation}\label{eq:attempt}
  T = \sum_{i=1}^{N} A_i + B,
\end{equation}
with $N \sim \mathrm{Geom}_0(p)$ counting failed attempts, $A_i$ i.i.d.\ failed-attempt lengths, and $B$ the successful-attempt length, is classical (see \S\ref{sec:attempt} for a self-contained derivation).  
The success probability $p = g(s)$, a single PGF evaluation, is the simplest instance of what we call \emph{proportional single-point sufficiency}.

The PGF of $T$ then takes the form (Theorem~\ref{thm:pgf})
\begin{equation}\label{eq:pgf-intro}
  G_T(w) \;=\; \frac{g(u)(1-u)}{1 - w + qw\,g(u)}, \qquad u = ws.
\end{equation}
This formula was obtained independently by~\cite{FP21} and~\cite{BP21}.  Our contribution begins with the observation that~\eqref{eq:pgf-intro} is \emph{affine in $g(u)$}:
\begin{equation}\label{eq:affine-intro}
  \underbrace{(1-u)}_{\alpha_N(w)} \cdot g(u) + \underbrace{0}_{\beta_N(w)}
  \quad\Big/\quad
  \underbrace{qw}_{\alpha_D(w)} \cdot g(u) + \underbrace{(1-w)}_{\beta_D(w)},
\end{equation}
where the argument $u = ws$ is \emph{linear} in $w$.  Since $\beta_N(1) = 0 = \beta_D(1)$, the probability axiom $G_T(1)=1$ forces
\begin{equation}\label{eq:coeff-match-intro}
  \alpha_N(1) = \alpha_D(1) = q.
\end{equation}

This coefficient equality is the organizing algebraic mechanism behind the derivative reduction principle and the subsequent approximation theory.

\subsection{Main Results}

We state our principal results in order: the coefficient equality~\eqref{eq:coeff-match-intro} drives the derivative reduction principle (Theorem~\ref{thm:A}), which in turn underpins the Laplace expansion (Theorem~\ref{thm:C}) and the sharp rate (Theorem~\ref{thm:D}).  The characterization (Theorem~\ref{thm:B}) closes the loop by showing that affine structure is not merely sufficient but also \emph{necessary} within the class of age-based catastrophe mechanisms.

\medskip
\noindent\textbf{I.\ Derivative reduction principle.}
The coefficient equality~\eqref{eq:coeff-match-intro} propagates through the Leibniz rule to produce an exact cancellation at every moment level.

\begin{theorem}[Derivative reduction principle]\label{thm:A}
Define $\Delta_k = N^{(k)}(1) - D^{(k)}(1)$, where $N$ and $D$ are the numerator and denominator of~\eqref{eq:pgf-intro}.  Then
\[
  \Delta_k =
  \begin{cases}
    1-p, & k=1, \\
    -k\,s^{k-1}\,D_{k-1}, & k \geq 2.
  \end{cases}
\]
In particular, $\Delta_k$ does not contain $D_k = g^{(k)}(s)$.  Consequently, the $k$-th factorial moment of $T$ depends on the $(k{-}1)$-jet $\{g(s), g'(s), \ldots, g^{(k-1)}(s)\}$, never on $g^{(k)}(s)$.
\end{theorem}

For details, see Theorem~\ref{thm:drp} and its proof. The theorem reveals a \emph{strict complexity staircase}: each moment level requires exactly one new derivative of the base PGF. For base processes with unbounded support, the staircase is tight---no further reduction is possible (Proposition~\ref{prop:tight}). In the coupon collector application (\S\ref{sec:ccp-app}), this staircase is realized by the finite harmonic sums $H^{(k-1)}=\sum_{j=m+1}^{n+m} j^{-(k-1)}$, so that the $k$-th moment introduces one new harmonic statistic that cannot be eliminated from the present recursion using lower-order data.

The cancellation rests on affine dependence on $g(u)$, a linear shift $u=ws$, and coefficient matching at the normalization point; see Lemma~\ref{lem:affine-ratio-rec}.  This raises a natural question: is their conjunction \emph{characteristic} of some identifiable class of catastrophe mechanisms?

\medskip
\noindent\textbf{II.\ Characterization of geometric-tail catastrophe.}
The answer is yes.  Consider a general age-based catastrophe mechanism with hazard sequence $(h_t)_{t \geq 1}$, $h_t \in (0,1)$, and survival function $S(t) = \prod_{i=1}^{t}(1-h_i)$.  We say the mechanism has \emph{geometric tail} if there exist $a \in (0,1)$ and $q \in (0,1)$ such that $h_1 = 1-a$ and $h_t = q$ for all $t \geq 2$; equivalently, $S(t) = a\,s^{t-1}$ for $t \geq 1$.  The special case $a = s$ recovers full memorylessness.

\begin{theorem}[Characterization]\label{thm:B}
Among age-based catastrophe mechanisms with hazard rates $h_t \in (0,1)$, the following are equivalent:
\begin{enumerate}[\rm (A)]
\item Geometric-tail catastrophe: $S(t) = a\,s^{t-1}$ for $t \geq 1$.
\item Uniform affine PGF structure (Definition~\ref{def:affine}).
\item Derivative reduction: there exists $z_0 \in (0,1)$ such that for every $k \geq 1$ and every finitely-supported~$T_0$, the $k$-th factorial moment of $T$ depends on $T_0$ only through $\{g^{(j)}(z_0)\}_{j=0}^{k-1}$.
\item Proportional single-point sufficiency: $p = \lambda\, g(z_0)$ for some $z_0 \in (0,1)$, $\lambda > 0$, and all finitely-supported~$T_0$.
\end{enumerate}
Full memorylessness ($h_t \equiv q$) is further equivalent to $\lambda = 1$ (Corollary~\ref{cor:full-memoryless}).
\end{theorem}

For details, see Theorem~\ref{thm:char}. The hardest implication (C)$\Rightarrow$(A) uses only the $k=1$ case and proceeds by a M\"obius rigidity argument (\S\ref{sec:char-thm}).  The gap between geometric-tail and full memorylessness is exactly one degree of freedom---the first-step hazard $h_1$.  In continuous time, this gap closes entirely (Theorem~\ref{thm:ct-char}); see~\textbf{VI} below.

\medskip
\noindent\textbf{III.\ Exponential limit and the Laplace expansion.}
The attempt decomposition~\eqref{eq:attempt} expresses $T$ as a geometric sum plus a perturbation.  R\'enyi's classical theorem~\cite{Renyi56} asserts that normalized geometric sums converge to the exponential distribution.  The coefficient equality~\eqref{eq:coeff-match-intro} sharpens this into a quantitative expansion.

\begin{theorem}[Laplace expansion]\label{thm:C}
Fix $q \in (0,1)$.  Let $W = T/E[T]$ and define
\[
  \alpha = \frac{s(p - qD_1)}{1-p}, \qquad \beta = \frac{1 - qp - qsD_1}{1-p}.
\]
Then $1 + \alpha - \beta = 0$ as an algebraic identity.  For $p$ sufficiently small and $\theta > 0$ satisfying $\theta p q / (1-p) \leq \epsilon_0$,
\[
  E[e^{-\theta W}] - \frac{1}{1+\theta}
  \;=\;
  \frac{\alpha\,\theta^2}{(1+\theta)^2}
  \;+\; O\!\left(\frac{|\alpha|^2\,\theta^3}{(1+\theta)^3}\right)
  \;+\; O\!\left(\frac{p\,\theta^2}{1+\theta}\right),
\]
with implicit constants depending only on $q$.  In particular, $T/E[T] \xrightarrow{d} \mathrm{Exp}(1)$ as $p \to 0$.
\end{theorem}

See Theorem~\ref{thm:laplace} for details. The identity $1 + \alpha - \beta = 0$ ensures that the first-order term in $\theta$ vanishes exactly---an algebraic consequence of the coefficient match~\eqref{eq:coeff-match-intro}, not an asymptotic cancellation.  The parameter $\alpha$ measures the deviation of the successful attempt from a geometric target (\S\ref{sec:alpha}).

\medskip
\noindent\textbf{IV.\ Sharp two-sided Kolmogorov bound.}
Converting the Laplace expansion to Kolmogorov distance via the smoothing inequality yields $d_K = O(|\alpha| \log(1/|\alpha|))$---a logarithmic artifact.  The following theorem eliminates this loss.

\begin{theorem}[Sharp Kolmogorov bound]\label{thm:D}
There exist constants $c_A, C_A, \delta_A > 0$ depending only on $q$ such that, for $p + |\alpha| \leq \delta_A$,
\[
  c_A\,(p + |\alpha|) \;\leq\; d_K\!\left(\frac{T}{E[T]},\, \mathrm{Exp}(1)\right) \;\leq\; C_A\,(p + |\alpha|).
\]
\end{theorem}

See Theorem~\ref{thm:two-sided} for details. The template $p + |\alpha|$ reflects two independent error sources: a \emph{lattice error} $\Theta(p)$ from the integrality of the attempt count, and a \emph{perturbation error} $\Theta(|\alpha|)$ from the successful attempt's length mismatch.  The \(p\)-term is irreducible, as shown by Proposition~\ref{prop:counterex}; the
\(|\alpha|\)-term is forced by the Laplace-transform lower bound and is dominant in the CCP regime.  The upper bound rests on a Bridge Theorem (Theorem~\ref{thm:bridge}) that decomposes the approximation via Brown's quantitative R\'enyi theorem (\cite{Brown90}); the lower bound combines a lattice argument with evaluation of the Laplace expansion.

\medskip
\noindent\textbf{V.\ CCP specialization and phase transition.}
We instantiate the theory to the coupon collector problem with $m$ reset coupons among $n+m$ types (\S\ref{sec:ccp-app}).  The specialization yields:
\begin{itemize}
\item Closed-form expressions for \(E[T]\) and \(\operatorname{Var}(T)\), together with an explicit harmonic realization of the DRP staircase for higher moments (\S\ref{sec:ccp-moments});
\item The CCP-specific bound $d_K \in [1-o(1),\, 2+o(1)] \cdot |\alpha|$ with $|\alpha| \sim m!\,m\ln n / n^m$ (Theorem~\ref{thm:ccp-sharp}), closing the logarithmic gap left by the smoothing approach;
\item A \emph{discontinuous Gumbel-to-Exponential phase transition} at $m = 0 \to m \geq 1$: one reset coupon suffices to replace extreme-value waiting ($(T_0 - n\ln n)/n \xrightarrow{d} \text{Gumbel}$) by geometric waiting ($T/E[T] \to$ Exponential).
\end{itemize}

\medskip
\noindent\textbf{VI.\ Continuous-time sharpening.}
Replacing PGFs by Laplace transforms and the multiplicative shift $u = ws$ by an additive shift $\sigma = \lambda + r$, the entire discrete theory carries over to continuous time (\S\ref{sec:continuous}), with one qualitative sharpening and one quantitative simplification.

The \emph{sharpening}: the geometric-tail gap of Theorem~\ref{thm:B} closes.  Among continuous-time age-based catastrophe mechanisms, affine Laplace structure characterizes \emph{full} Poisson resetting---constant hazard everywhere, with no residual first-step degree of freedom (Theorem~\ref{thm:ct-char}).  The mechanism is twofold: point masses $T_0 = \delta_x$ with continuous $x$ determine the function $\psi$ on all of $(0,1)$, and the relation $F = \rho\,S + c$ is \emph{differentiated} rather than \emph{differenced}, constraining the hazard rate at every point.

The \emph{simplification} is conditional rather than unconditional: In continuous time, the lattice error disappears. In the regime $|\alpha|\gg p$, the sharp template reduces to $|\alpha|$; when $\alpha$ is small, a residual $O(p)$ term may remain. In the Brownian first-passage model of~\cite{EM11}, this yields $\dK\!\left(\frac{T}{E[T]},\,\Exp(1)\right)\in [1-o(1),\,2+o(1)]\cdot |\alpha|$ (\S\ref{sec:brownian}).

\subsection{Architecture of the Paper}

The paper is organized as follows.  \S\ref{sec:completion} derives the PGF formula, identifies the affine structure, and establishes the coefficient equality. \S\ref{sec:drp} proves the derivative reduction principle and the moment recursion.  \S\ref{sec:char} establishes the characterization theorem via a M\"obius rigidity argument.  \S\ref{sec:exp-approx} develops the exponential approximation: the Laplace expansion, the Bridge Theorem, and the sharp two-sided Kolmogorov bound $d_K \asymp p + |\alpha|$, including a counterexample showing that neither term is redundant.  \S\ref{sec:continuous} presents the continuous-time counterpart---replacing PGFs by Laplace transforms and the multiplicative shift by an additive one---and shows that the geometric-tail gap closes: affine Laplace structure characterizes full Poisson resetting.  \S\ref{sec:ccp-app} and \S\ref{sec:multiphase} instantiate the theory to two structurally contrasting applications: the coupon collector with reset coupons, which produces a harmonic moment staircase and a Gumbel-to-Exponential phase transition, and a multi-phase task under catastrophe, which produces an algebraic staircase and a Gaussian-to-exponential transition with exponentially fast convergence.  \S\ref{sec:related} discusses related work, and \S\ref{sec:further} outlines further directions.

\section{Completion Time under Memoryless Catastrophe}\label{sec:completion}

\subsection{Model and PGF Formula}\label{sec:attempt}\label{sec:model}

Let $T_0$ be a positive-integer-valued, almost surely finite random variable, called the \emph{base completion time}, with $\pi_j = P(T_0 = j)$ for $j \geq 1$ and PGF $g(z) = G_{T_0}(z) = \sum_{k \geq 0} P(T_0=k)\, z^k$ for $|z| \leq 1$.

\begin{definition}[Memoryless catastrophe]\label{def:memoryless}
Fix $q \in (0,1)$ and set $s = 1-q$.  At each step, independently of everything else, a catastrophe occurs with probability $q$.  If the process survives (probability~$s$), the base task may advance; if a catastrophe occurs and the base task has not yet completed, all accumulated progress is destroyed and the process restarts from scratch.  The process terminates at the first step where it survives the catastrophe and the base task completes.
\end{definition}

A maximal run of steps uninterrupted by catastrophe is called an \emph{attempt}.  Each attempt begins from the empty state and either \emph{succeeds} (the base task completes) or \emph{fails} (terminated by catastrophe).  Since the base task completes at step $t$ with no prior catastrophe with probability $s^t\,\pi_t$ (independence), the success probability is\label{lem:success}
\begin{equation}\label{eq:p}
  p \;=\; \sum_{t \geq 1} s^t\, \pi_t \;=\; g(s) \;\in\; (0,1).
\end{equation}
Successive attempts are independent (each restarts from the empty state), so the number of failed attempts before the first success is $N \sim \mathrm{Geom}_0(p)$.\footnote{We write $\mathrm{Geom}_0(p)$ for the distribution on $\{0,1,2,\ldots\}$ with $P(N=k)=(1-p)^k p$ (failure count), and $\mathrm{Geom}_1(p) = \mathrm{Geom}_0(p)+1$ for the trial count.}\label{def:geom}

Denote by $A$ the length of a generic failed attempt, by $B$ the length of the successful attempt, and let $A_1, A_2, \ldots$ be i.i.d.\ copies of $A$.  The total completion time decomposes as
\begin{equation}\label{eq:decomp}
  T \;=\; \sum_{i=1}^{N} A_i \;+\; B,
\end{equation}
where $N$, $(A_i)_{i \geq 1}$, and $B$ are mutually independent.  Standard conditioning (Bayes' rule on the success/failure event) yields the PGFs\label{lem:cond-pgf}
\begin{align}
  G_B(w) &= \frac{g(ws)}{p}, \label{eq:GB} \\[4pt]
  G_A(w) &= \frac{qw\bigl(1 - g(ws)\bigr)}{(1-p)(1-ws)}, \label{eq:GA}
\end{align}
where~\eqref{eq:GA} uses the identity $\sum_{j \geq 0} z^j P(T_0 > j) = (1 - g(z))/(1-z)$.  Combining these with the geometric-sum PGF $E[G_A(w)^N] = p/(1-(1-p)G_A(w))$ gives the following.

\begin{theorem}[PGF formula]\label{thm:pgf}
Under memoryless catastrophe with rate $q$, the PGF of the completion time is
\begin{equation}\label{eq:pgf}
  G_T(w) \;=\; \frac{g(u)(1-u)}{1 - w + qw\,g(u)}, \qquad u = ws.
\end{equation}
\end{theorem}

\begin{proof}
By~\eqref{eq:decomp} and mutual independence, $G_T = G_B \cdot E[G_A^N]$.  Substituting~\eqref{eq:GB} and~\eqref{eq:GA}:
\[
  G_T(w) = \frac{g(u)}{p} \cdot \frac{p}{1 - \dfrac{qw(1-g(u))}{1-u}}
  = \frac{g(u)(1-u)}{(1-u) - qw(1-g(u))}.
\]
The denominator simplifies: $(1-ws) - qw + qw\,g(u) = 1 - w(s+q) + qw\,g(u) = 1 - w + qw\,g(u)$, using $s + q = 1$.
\end{proof}

\begin{corollary}[Expected completion time]\label{cor:ET}
$\displaystyle E[T] = \frac{1-p}{qp}$.
\end{corollary}

\begin{proof}
Differentiating~\eqref{eq:pgf}: $G_T(w)-1 = (1-w)(g(u)-1)/(1-w+qwg(u))$ (using $1-u-qw = 1-w$), so $E[T] = \lim_{w \to 1}(G_T(w)-1)/(w-1) = (1-p)/(qp)$.
\end{proof}


\subsection{Affine Structure and the Derivative Reduction Principle}\label{sec:affine}

We now identify the structural property of~\eqref{eq:pgf} that drives all subsequent results.

\begin{definition}[Uniform affine PGF structure]\label{def:affine}
A catastrophe mechanism has uniform affine PGF structure if there exist functions
\(a(w),b(w),c(w),d(w)\), analytic in a neighborhood of \(w=1\), depending only on the mechanism
(not on the base process \(T_0\)), and a linear function \(u(w)\) with \(u(1)\in(0,1)\), such that for all base processes~$T_0$,
\[
  G_T(w) = \frac{a(w)\,g(u(w)) + b(w)}{c(w)\,g(u(w)) + d(w)},
\]
with $a$ and $c$ not identically zero, and the non-degeneracy condition $c(1)\,g(u(1)) + d(1) \neq 0$ for all admissible~$g$.\label{rem:nondegen}\footnote{Non-degeneracy excludes pathological representations obtained by multiplying numerator and denominator by a factor vanishing at $w = 1$, which would create a $0/0$ indeterminacy and invalidate the coefficient-matching argument.}
\end{definition}

\begin{proposition}[Affine structure and coefficient equality]\label{prop:affine}
The PGF formula~\eqref{eq:pgf} exhibits uniform affine structure with
\[
  a(w) = 1-u, \quad b(w) = 0, \quad c(w) = qw, \quad d(w) = 1-w, \quad u(w) = ws.
\]
Writing $N(w) = a(w)\,g(u) + b(w)$ and $D(w) = c(w)\,g(u) + d(w)$ for the numerator and denominator, the normalization $G_T(1) = 1$ forces
\begin{equation}\label{eq:coeff-match}
  \alpha_N(1) \;:=\; a(1) \;=\; q \;=\; c(1) \;=:\; \alpha_D(1).
\end{equation}
\end{proposition}

\begin{proof}
Immediate from~\eqref{eq:pgf}:  at $w = 1$, $b(1) = 0 = d(1)$, so $G_T(1) = 1$ gives $a(1)\,g(s) = c(1)\,g(s)$ and $a(1) = c(1)$ (since $g(s) = p > 0$); since $a(1) = 1-s = q$, the result follows.
\end{proof}

The coefficient equality~\eqref{eq:coeff-match} is a direct consequence of the probability axiom $G_T(1) = 1$ applied to the affine form.  Together with the linearity of $u(w) = ws$ and the affine dependence on $g(u)$, it produces an exact derivative cancellation at every moment level.

\medskip

The PGF formula expresses $G_T(w) = N(w)/D(w)$, where\label{sec:drp}
\begin{equation}\label{eq:ND-def}
  N(w) = (1-u)\,g(u), \qquad D(w) = (1-w) + qw\,g(u), \qquad u = ws.
\end{equation}
The $k$-th factorial moment is $\eta_k = G_T^{(k)}(1)$.  Since $u = ws$ is linear, $\frac{d^j}{dw^j} g(ws) = s^j g^{(j)}(ws)$ with no Fa\`a di Bruno corrections.  Writing $D_j = g^{(j)}(s)$, one might expect $\eta_k$ to depend on the full $k$-jet $\{D_0, \ldots, D_k\}$.  The following theorem shows that $D_k$ cancels exactly.

\begin{theorem}[Derivative reduction principle]\label{thm:drp}
Define $\Delta_k = N^{(k)}(1) - D^{(k)}(1)$ for $k \geq 1$.  Then
\begin{equation}\label{eq:Delta}
  \Delta_k =
  \begin{cases}
    1-p, & k=1, \\[3pt]
    -k\,s^{k-1}\,D_{k-1}, & k \geq 2.
  \end{cases}
\end{equation}
In particular, $\Delta_k$ does not contain $D_k = g^{(k)}(s)$.
\end{theorem}

\begin{proof}
We compute $N^{(k)}(1)$ and $D^{(k)}(1)$ separately via the Leibniz rule, then subtract.

\medskip
\noindent\emph{Case $k \geq 2$.}\;
Write $N(w) = f_1(w) \cdot g(ws)$ with $f_1(w) = 1 - ws$.  Since $f_1(1) = q$, $f_1'(w) = -s$, and $f_1^{(j)} = 0$ for $j \geq 2$, the Leibniz rule gives
\begin{equation}\label{eq:Nk}
  N^{(k)}(1) = q\,s^k D_k - k\,s^k D_{k-1}.
\end{equation}
Similarly, $D(w) = (1-w) + f_2(w) \cdot g(ws)$ with $f_2(w) = qw$.  For $k \geq 2$, $\frac{d^k}{dw^k}(1-w) = 0$.  Since $f_2(1) = q$, $f_2'(w) = q$, and $f_2^{(j)} = 0$ for $j \geq 2$:
\begin{equation}\label{eq:Dk}
  D^{(k)}(1) = q\,s^k D_k + k\,q\,s^{k-1} D_{k-1}.
\end{equation}
Subtracting~\eqref{eq:Dk} from~\eqref{eq:Nk}:
\[
  \Delta_k = \underbrace{(q\,s^k - q\,s^k)}_{= 0}\,D_k
  + \bigl(-k\,s^k - k\,q\,s^{k-1}\bigr)\,D_{k-1}
  = -k\,s^{k-1}(s+q)\,D_{k-1}
  = -k\,s^{k-1} D_{k-1},
\]
where the last step uses $s + q = 1$.  The $D_k$-terms cancel because they enter $N^{(k)}(1)$ and $D^{(k)}(1)$ with the same coefficient $q\,s^k$---precisely the coefficient equality~\eqref{eq:coeff-match}.

\medskip
\noindent\emph{Case $k = 1$.}\;
Direct computation: $N'(1) = -sp + qsD_1$, $D'(1) = -1 + qp + qsD_1$.  Hence
\[
  \Delta_1 = (-sp + qsD_1) - (-1 + qp + qsD_1) = 1 - p(s+q) = 1 - p. \qedhere
\]
\end{proof}

The next lemma extracts the cancellation mechanism behind Theorem~\ref{thm:drp}.

\begin{lemma}[Affine-ratio recursion]\label{lem:affine-ratio-rec}
Let
\[
  R(z)=\frac{a(z)h(u(z))+b(z)}{c(z)h(u(z))+d(z)},
  \qquad
  u(z)=u_0+\kappa(z-z_*),
\]
where $a,b,c,d$ are $C^\infty$ near $z_*$, and $a(z_*)=c(z_*), b(z_*)=d(z_*), D(z_*):=c(z_*)h(u_0)+d(z_*)\neq 0$.
Set
\[
  N(z)=a(z)h(u(z))+b(z),\qquad D(z)=c(z)h(u(z))+d(z).
\]
Then, for every $k\ge 1$, $\Delta_k:=N^{(k)}(z_*)-D^{(k)}(z_*)$ depends only on $h(u_0),h'(u_0),\dots,h^{(k-1)}(u_0)$, and
\[
  R^{(k)}(z_*)
  =
  \frac{1}{D(z_*)}
  \left(
    \Delta_k-\sum_{j=1}^{k-1}\binom{k}{j}D^{(j)}(z_*)R^{(k-j)}(z_*)
  \right).
\]
Consequently, $R^{(k)}(z_*)$ depends only on the $(k-1)$-jet of $h$ at $u_0$.
\end{lemma}

\begin{proof}
Since $u$ is linear, $(h\circ u)^{(k)}(z_*)=\kappa^k h^{(k)}(u_0)$ and no Fa\`a di Bruno terms appear. Hence the coefficient of $h^{(k)}(u_0)$ in $N^{(k)}(z_*)$ is $a(z_*)\kappa^k$, while in $D^{(k)}(z_*)$ it is $c(z_*)\kappa^k$; these agree by $a(z_*)=c(z_*)$, so $\Delta_k$ contains no $h^{(k)}(u_0)$. Differentiating the identity $D(z)R(z)=N(z)$ $k$ times at $z=z_*$ and isolating the $j=0$ and $j=k$ terms gives the displayed recursion.
\end{proof}


\subsection{Moment Recursion and the Complexity Staircase}\label{sec:staircase}

\begin{corollary}[Moment recursion]\label{cor:moment-recursion}
The factorial moments of $T$ satisfy $\eta_0 = 1$ and, for $k \geq 1$,
\begin{equation}\label{eq:recursion}
  \eta_k = \frac{1}{pq}\left(\Delta_k - \sum_{j=1}^{k-1} \binom{k}{j} D^{(j)}(1)\,\eta_{k-j}\right),
\end{equation}
where $D^{(j)}(1)$ is the $j$-th derivative of $D(w) = (1-w) + qw\,g(ws)$ at $w = 1$.  Since $\Delta_k$ involves at most $D_{k-1}$ (Theorem~\ref{thm:drp}) and $D^{(j)}(1)$ involves at most $D_j$ for $j \leq k-1$ (by~\eqref{eq:Dk}), induction on $k$ confirms: $\eta_k$ depends on the $(k{-}1)$-jet $\{D_0, D_1, \ldots, D_{k-1}\}$, never on $D_k$.
\end{corollary}

\begin{proof}
Apply the Leibniz rule to $D \cdot G_T = N$ at $w = 1$, separate the $j = 0$ and $j = k$ terms, and use $\Delta_k = N^{(k)}(1) - D^{(k)}(1)$ to cancel $D^{(k)}(1)$.
\end{proof}

The low-order derivatives of $D$ at $w = 1$ are:
\begin{equation}\label{eq:D-derivs}
  D(1) = qp, \qquad
  D'(1) = -1 + qp + qsD_1, \qquad
  D''(1) = 2qsD_1 + qs^2 D_2.
\end{equation}

\begin{example}[First two moments]\label{ex:moments}
From~\eqref{eq:recursion} with $k = 1$: $\eta_1 = (1-p)/(pq)$, recovering Corollary~\ref{cor:ET}.  With $k = 2$: $\eta_2 = \frac{1}{pq}(-2sD_1 - 2D'(1)\,\eta_1)$.  The variance $\Var(T) = \eta_2 + \eta_1 - \eta_1^2$ simplifies to
\begin{equation}\label{eq:var}
  \Var(T) = \frac{(1-p)(1+ps) - 2qsD_1}{p^2 q^2},
\end{equation}
which depends on $p = D_0$ and $D_1 = g'(s)$ but not on $D_2$---exactly as the DRP predicts.
\end{example}

The DRP asserts that each moment level requires at most $D_{k-1}$.  The following shows that, for base processes with sufficiently rich structure, it requires \emph{exactly} $D_{k-1}$.

\begin{proposition}[Tightness of the staircase]\label{prop:tight}
For $k \geq 2$, the total coefficient of $D_{k-1}$ in $\eta_k$ is $-ks^{k-1}/(p^2 q) \neq 0$, provided $D_{k-1} \neq 0$.  The latter holds whenever $P(T_0 \geq k) > 0$.  In particular, for base processes with unbounded support, every moment level introduces exactly one new derivative of $g$.
\end{proposition}

\begin{proof}
By Theorem~\ref{thm:drp}, $\Delta_k = -ks^{k-1} D_{k-1}$ for $k \geq 2$.  By induction, $\eta_{k-j}$ for $j \geq 1$ depends on $\{D_0, \ldots, D_{k-j-1}\}$, so $D_{k-1}$ enters the sum in~\eqref{eq:recursion} only through the $j = k-1$ term $D^{(k-1)}(1)\,\eta_1$, where $D^{(k-1)}(1)$ contains $D_{k-1}$ with coefficient $qs^{k-1}$.  Combining: the total coefficient of $D_{k-1}$ in $\eta_k$ is $\frac{1}{pq}(-ks^{k-1} - k\,qs^{k-1}\,\eta_1) = \frac{-ks^{k-1}}{pq}(1 + q\eta_1) = \frac{-ks^{k-1}}{p^2 q}$, using $\eta_1 = (1-p)/(pq)$.

For the second claim: $g^{(k-1)}(s) = \sum_{t \geq k} \frac{t!}{(t-k+1)!} \pi_t s^{t-k+1} > 0$ whenever some $\pi_t > 0$ for $t \geq k$.
\end{proof}

Together, Theorem~\ref{thm:drp} and Proposition~\ref{prop:tight} establish a sharp complexity staircase: for base processes with unbounded support, each moment of $T$ introduces exactly one new derivative of $g$, and this staircase collapses only when $T_0$ has bounded support.

\section{Characterization of Geometric-Tail Catastrophe}\label{sec:char}

\S\ref{sec:completion} established a chain of implications under memoryless catastrophe:
\[
  \text{memoryless catastrophe}
  \;\Longrightarrow\;
  \text{affine PGF structure}
  \;\Longrightarrow\;
  \text{derivative reduction (DRP)}.
\]
A natural question arises: \emph{is memoryless catastrophe the only mechanism producing these properties, or do they hold more broadly?}  This section answers the question precisely.  The answer is that affine PGF structure---and hence the DRP---characterizes a slightly larger class, which we call \emph{geometric-tail catastrophe}: mechanisms whose hazard rate is constant from the second step onward, with the first-step hazard left free.  Full memorylessness is the unique member of this class satisfying an additional algebraic condition ($p = g(s)$ rather than $p \propto g(s)$).

To state and prove this characterization, we first extend the framework of \S\ref{sec:completion} to general age-based catastrophe mechanisms.


\subsection{General Hazard Framework}\label{sec:hazard}

\begin{definition}[Age-based catastrophe mechanism]\label{def:hazard}
An \emph{age-based catastrophe mechanism} is specified by a hazard sequence $(h_t)_{t \geq 1}$ with $h_t \in (0,1)$ for all $t$.  Here $h_t$ is the conditional probability of catastrophe at step $t$, given that no catastrophe has occurred in steps $1, \ldots, t-1$.  The associated \emph{survival function} is
\begin{equation}\label{eq:survival}
  S(0) = 1, \qquad S(t) = \prod_{i=1}^{t}(1 - h_i), \quad t \geq 1.
\end{equation}
Thus $S(t)$ is the probability of surviving $t$ consecutive steps without catastrophe.  The requirement $h_t \in (0,1)$ ensures $S(t) > 0$ for all $t$ (every step has a positive probability of both survival and catastrophe).
\end{definition}

Note that memoryless catastrophe (Definition~\ref{def:memoryless}) is the special case $h_t \equiv q$, giving $S(t) = s^t$. The attempt decomposition~\eqref{eq:decomp} remains valid under any age-based mechanism: each attempt starts from the empty state, so successive attempts are i.i.d.  The success probability generalizes from $p = g(s)$ to
\begin{equation}\label{eq:p-general}
  p = \sum_{t \geq 1} S(t)\,\pi_t,
\end{equation}
which is no longer a PGF evaluation unless $S(t)$ is exponential in $t$.  The expected completion time generalizes from Theorem~\ref{thm:pgf} as follows.

\begin{lemma}[Expected attempt length]\label{lem:attempt-length}
Under an age-based mechanism with survival function $S$, the expected length of a single attempt (regardless of outcome) satisfies
\begin{equation}\label{eq:ell}
  E[\ell] = \sum_{j \geq 1} F(j)\,\pi_j, \qquad F(j) = \sum_{t=0}^{j-1} S(t).
\end{equation}
\end{lemma}

\begin{proof}
Condition on $T_0 = j$: step $t \leq j$ is reached with probability $S(t-1)$.
\end{proof}

\begin{proposition}[General expected completion time]\label{prop:ET-general}
Assume $E[\ell]=\sum_{j\ge1}F(j)\pi_j<\infty$. Then
\begin{equation}\label{eq:ET-general}
  E[T]=\frac{E[\ell]}{p}
  =\frac{\sum_{j\ge1}F(j)\pi_j}{\sum_{j\ge1}S(j)\pi_j}.
\end{equation}
\end{proposition}

\begin{proof}
The pairs $(\ell_i,\mathrm{outcome}_i)$ are i.i.d., and $M=N+1\sim \Geom_1(p)$ is a stopping time. Since $E[\ell]<\infty$, Wald's identity applies and gives $E[T]=E[\ell]\cdot E[M]=\frac{E[\ell]}{p}$.
\end{proof}

Formula~\eqref{eq:ET-general} expresses $E[T]$ as a ratio of two linear functionals of the base distribution $(\pi_j)_{j \geq 1}$.  In the memoryless case, $S(t) = s^t$ makes the denominator $p = g(s)$---a PGF evaluation---and $E[T]$ depends on $T_0$ only through this single number.  For a general mechanism, $E[T]$ depends on the full interaction between $(\pi_j)$ and $S$.

\subsection{Geometric-Tail Catastrophe}\label{sec:geom-tail}

\begin{definition}[Geometric-tail catastrophe]\label{def:geom-tail}
A hazard sequence $(h_t)_{t \geq 1}$ has \emph{geometric tail} if there exist $a \in (0,1)$ and $q \in (0,1)$ such that
\[
  h_1 = 1-a, \qquad h_t = q \quad \text{for all } t \geq 2.
\]
Equivalently, $S(t) = a\,s^{t-1}$ for $t \geq 1$, where $s = 1-q$.  The parameter $a = S(1) = 1-h_1$ is the first-step survival probability.  The special case $a = s$ (i.e., $h_1 = q$) recovers memoryless catastrophe.
\end{definition}

The geometric-tail class has one degree of freedom beyond full memorylessness: the first-step hazard $h_1$ may differ from $q$.  The next proposition shows that every member of this class produces a PGF with affine structure, establishing the forward direction of the characterization.

\begin{proposition}[Geometric-tail PGF formula]\label{prop:geom-tail-pgf}
Under geometric-tail catastrophe with parameters $(a, q)$, the success probability is $p = (a/s)\,g(s)$ and the completion-time PGF is
\begin{equation}\label{eq:geom-tail-pgf}
  G_T(w) = \frac{a(1-u)\,g(u)}{s(1-w)(1 - w(s-a)) + aqw\,g(u)}, \qquad u = ws.
\end{equation}
This exhibits uniform affine PGF structure (Definition~\ref{def:affine}) with coefficient equality $\alpha_N(1) = \alpha_D(1) = aq$.  Setting $a = s$ recovers the memoryless formula~\eqref{eq:pgf}.
\end{proposition}

\begin{proof}
Since $S(t) = as^{t-1}$, the success probability is $p = (a/s)\,g(s)$ and the successful-attempt PGF is $G_B(w) = a\,g(u)/(ps)$, where $u = ws$.

For the failed-attempt PGF, note that the only departure from the memoryless calculation (\S\ref{sec:completion}) is the first-step hazard $h_1 = 1-a \neq q$.  Splitting accordingly:
\[
  (1-p)\,G_A(w) = w(1-a) + \frac{aqw}{s}\sum_{j \geq 1}u^j\,P(T_0 > j).
\]
The identity $\sum_{j \geq 1}u^j P(T_0 > j) = (u - g(u))/(1-u)$ (used in \S\ref{sec:completion}) gives
\[
  (1-p)\,G_A(w) = w(1-a) + \frac{aqw(u - g(u))}{s(1-u)}.
\]
Substituting into $G_T = p\,G_B/[1-(1-p)G_A]$ and clearing denominators, the denominator simplifies (using $s + q = 1$ and $u = ws$) to $s(1-w)(1-w(s-a)) + aqw\,g(u)$, yielding~\eqref{eq:geom-tail-pgf}.  (Setting $a = s$ eliminates the $1-w(s-a)$ factor and recovers~\eqref{eq:pgf}.)

The $g(u)$-coefficients are $\alpha_N(w) = a(1-u)$ and $\alpha_D(w) = aqw$, giving $\alpha_N(1) = \alpha_D(1) = aq$.  Since the numerator has no $g$-free term ($\beta_N \equiv 0$) and $\beta_D(1) = 0$, we obtain $G_T(1) = 1$.  Non-degeneracy: $D(1) = aqg(s) = pqs > 0$.
\end{proof}

Since the geometric-tail PGF has the same affine-ratio form, the same linear shift $u=ws$, and coefficient matching at $w=1$, Lemma~\ref{lem:affine-ratio-rec} applies verbatim. Hence the DRP holds for geometric-tail mechanisms, with $\alpha_N(1)=\alpha_D(1)=aq$ replacing $q$.

\subsection{The Characterization Theorem}\label{sec:char-thm}

We now prove the converse: geometric-tail catastrophe is the \emph{only} age-based mechanism producing affine PGF structure.  The result takes the form of a four-way equivalence linking the survival function, the PGF structure, the derivative reduction property, and the success probability formula.

\begin{theorem}[Characterization]\label{thm:char}
Among age-based catastrophe mechanisms (Definition~\ref{def:hazard}), the following are equivalent:
\begin{enumerate}[\rm (A)]
  \item \emph{Geometric-tail catastrophe}: $S(t) = a\,s^{t-1}$ for $t \geq 1$ (Definition~\ref{def:geom-tail}).
  \item \emph{Uniform affine PGF structure}: Definition~\ref{def:affine}.
  \item \emph{Derivative reduction}: there exists $z_0 \in (0,1)$ such that for every $k \geq 1$ and every finitely-supported $T_0$, the $k$-th factorial moment $\eta_k$ depends on $T_0$ only through $\{g^{(j)}(z_0)\}_{j=0}^{k-1}$.
  \item \emph{Proportional single-point sufficiency}: there exist $z_0 \in (0,1)$ and $\lambda > 0$ such that $p = \lambda\,g(z_0)$ for all finitely-supported $T_0$.
\end{enumerate}
\end{theorem}

We prove the cycle (A)$\Rightarrow$(B)$\Rightarrow$(C)$\Rightarrow$(A), and independently prove (A)$\Leftrightarrow$(D).  The implications (A)$\Rightarrow$(B) and (A)$\Rightarrow$(D) have been established; the substance lies in (C)$\Rightarrow$(A), which uses only the $k = 1$ instance of (C) and proceeds by a M\"obius rigidity argument.

\begin{proof}[Proof of {\rm (A)$\Rightarrow$(B)}]
This is Proposition~\ref{prop:geom-tail-pgf}.
\end{proof}

\begin{proof}[Proof of {\rm (B)$\Rightarrow$(C)}]
Suppose
\[
  G_T(w)=\frac{a(w)g(u(w))+b(w)}{c(w)g(u(w))+d(w)},
  \qquad u(1)=u_1\in(0,1),
\]
with $u$ linear. From $G_T(1)=1$ for all finitely-supported $T_0$, taking $T_0\equiv n$ and letting $n\to\infty$ gives $b(1)=d(1)$, and then back-substitution gives $a(1)=c(1)$. Since non-degeneracy yields $c(1)g(u_1)+d(1)\neq 0$, Lemma~\ref{lem:affine-ratio-rec} applies with $R=G_T$, $h=g$, $z_*=1$, and $u_0=u_1$. Therefore the $k$-th factorial moment $\eta_k=G_T^{(k)}(1)$ depends only on $\{g^{(j)}(u_1)\}_{j=0}^{k-1}$.
\end{proof}

\begin{proof}[Proof of {\rm (C)$\Rightarrow$(A)}]
We use only the $k = 1$ case: $E[T]$ depends on $T_0$ only through $g(z_0)$ for some $z_0 \in (0,1)$.  By Proposition~\ref{prop:ET-general}, there exists $\psi: (0,1) \to \mathbb{R}$ such that
\begin{equation}\label{eq:psi-hyp}
  \frac{\sum_{j \geq 1} F(j)\,\pi_j}{\sum_{j \geq 1} S(j)\,\pi_j}
  = \psi\!\left(\sum_{j \geq 1} z_0^j\,\pi_j\right)
  \quad\text{for all finitely-supported } (\pi_j).
\end{equation}

\emph{Step 1: Point masses.}\;
Taking $\pi = \delta_j$ gives $\psi(z_0^j) = F(j)/S(j) =: f_j$ for each $j \geq 1$.

\emph{Step 2: Two-point distributions force $\psi$ to be M\"obius.}\;
For $\pi_a = \alpha$, $\pi_b = 1-\alpha$ with $a < b$, set $v = \alpha z_0^a + (1-\alpha)z_0^b$.  Substituting into~\eqref{eq:psi-hyp} and solving for $\alpha$ in terms of $v$ gives
\begin{equation}\label{eq:mobius-local}
  \psi(v) = \frac{(F(a)-F(b))\,v + F(b)z_0^a - F(a)z_0^b}{(S(a)-S(b))\,v + S(b)z_0^a - S(a)z_0^b}
  =: M_{ab}(v)
\end{equation}
on $[z_0^b, z_0^a]$.  This is a M\"obius transformation of $v$.

\emph{Step 3: Global consistency.}\;
For $a < c < b$, the M\"obius transformations $M_{ab}$ and $M_{ac}$ agree on the interval $[z_0^c, z_0^a]$, which contains infinitely many points; since a non-degenerate M\"obius transformation is determined by three values, $M_{ab} = M_{ac}$ as rational functions. For any two pairs $(a_1,b_1)$, $(a_2,b_2)$ with $a_i<b_i$, setting $a^*=\min(a_1,a_2)$, $b^*=\max(b_1,b_2)$ gives $[z_0^{b_i},z_0^{a_i}] \subset [z_0^{b^*},z_0^{a^*}]$ for each $i$ (since $z_0\in(0,1)$ makes $j\mapsto z_0^j$ decreasing), so $M_{a_1 b_1} = M_{a^*b^*} = M_{a_2 b_2}$.  Hence all $M_{ab}$ coincide with a single 
\begin{equation}\label{eq:global-mobius}
  \psi(v) = M(v) = \frac{Av + B}{Cv + D} \qquad \text{on } (0, z_0].
\end{equation}
The denominator slope of each $M_{ab}$ is proportional to $S(a) - S(b) \neq 0$ (since $h_t \in (0,1)$ makes $S$ strictly decreasing), so $C \neq 0$.

\emph{Step 4: Coefficient consistency forces $F = rS + c$.}\;
Since $M_{ab} = M$ up to a common scalar, $A/C = (F(a)-F(b))/(S(a)-S(b))$ is constant across all $a \neq b$.  Denoting this constant by $r$, we obtain $F(j) = rS(j) + c$ for all $j \geq 1$.

\emph{Step 5: Differencing yields constant hazard from step $2$ onward.}\;
From $F(j) = \sum_{t=0}^{j-1}S(t)$: $F(j+1) - F(j) = S(j)$.  From $F = rS + c$: $F(j+1) - F(j) = r(S(j+1) - S(j)) = -rS(j)h_{j+1}$.  Equating and dividing by $S(j) > 0$:
\[
  h_{j+1} = -1/r \quad\text{for all } j \geq 1.
\]
Hence $h_t = q := -1/r$ for all $t \geq 2$, with $r < -1$ (since $q \in (0,1)$).  Setting $a = S(1) \in (0,1)$ gives $S(t) = as^{t-1}$ for $t \geq 1$, confirming geometric-tail catastrophe.
\end{proof}

\begin{proof}[Proof of {\rm (A)$\Rightarrow$(D)}]
If $S(t) = as^{t-1}$, then $p = \sum_{t \geq 1}as^{t-1}\pi_t = (a/s)\,g(s)$, with $z_0 = s$ and $\lambda = a/s > 0$.
\end{proof}

\begin{proof}[Proof of {\rm (D)$\Rightarrow$(A)}]
The hypothesis $p = \lambda\,g(z_0)$ for all finitely-supported $(\pi_j)$ gives, upon taking $\pi = \delta_t$: $S(t) = \lambda\,z_0^t$ for each $t \geq 1$.  Hence $h_t = 1 - S(t)/S(t-1) = 1 - z_0$ for all $t \geq 2$, confirming geometric tail with $s = z_0$, $q = 1-z_0$, and $a = \lambda z_0$.
\end{proof}


\begin{remark}[Emphasis of the (C)$\Rightarrow$(A) argument]\label{rem:mobius}
Only the $k=1$ instance of~(C) is needed: the condition that $E[T]$ depends on $T_0$ through $g(z_0)$ alone already forces geometric-tail catastrophe.  The mechanism is M\"obius rigidity: the ratio of two linear functionals of $(\pi_j)$ equaling a function of a third forces that function to be linear fractional, and determination by three points propagates the local constraint globally.  This is the same low-dimensionality phenomenon that drives the DRP---there, affine dependence on $g(u)$ forces the highest derivative to cancel; here, it forces the hazard sequence to stabilize.
\end{remark}

\subsection{From Geometric Tail to Full Memorylessness}\label{sec:full-memoryless}

The geometric-tail class has one free parameter ($a = S(1)$) beyond the catastrophe rate $q$.  The following corollary identifies the algebraic condition that pins down this parameter.

\begin{corollary}[Full memorylessness]\label{cor:full-memoryless}
Among geometric-tail mechanisms, full memorylessness ($h_t \equiv q$ for all $t \geq 1$) is equivalent to $\lambda = 1$ in condition~(D) of Theorem~\ref{thm:char}, i.e., the success probability is a literal PGF evaluation:
\[
  p = g(s).
\]
\end{corollary}

\begin{proof}
Geometric tail gives $p = (a/s)\,g(s)$, so $\lambda = a/s$.  If $\lambda = 1$ then $a = s$, hence $h_1 = 1-a = 1-s = q$, giving full memorylessness.  Conversely, memorylessness gives $a = s$ and $\lambda = 1$.
\end{proof}

\begin{example}[Geometric-tail but not memoryless]\label{ex:geom-tail}
Set $h_1 = 1/2$ and $h_t = 1/4$ for $t \geq 2$, so $a = 1/2$, $q = 1/4$, $s = 3/4$.  Then $S(t) = (1/2)(3/4)^{t-1}$ and $p = (2/3)\,g(3/4)$.  Since $h_1 = 1/2 \neq 1/4 = q$, this is not memoryless, yet it has uniform affine PGF structure by Proposition~\ref{prop:geom-tail-pgf}, with $\alpha_N(1) = \alpha_D(1) = aq = 1/8$.  This confirms that affine PGF structure characterizes the geometric-tail class, not full memorylessness.
\end{example}

\begin{remark}[Proportionality is essential]\label{rem:proportionality}
Condition~(D) requires $p = \lambda\,g(z_0)$---proportionality, not merely functional dependence.  If one weakens~(D) to ``$p = \varphi(g(z_0))$ for some function $\varphi$,'' then two-point distributions force $\varphi$ to be affine: $\varphi(x) = \lambda x + \mu$.  When $\mu > 0$, the resulting survival function $S(j) = \lambda z_0^j + \mu$ has $h_j \to 0$ as $j \to \infty$, violating both~(B) and~(C).  Thus proportionality ($\mu = 0$) is necessary for the equivalence.
\end{remark}

\section{Exponential Approximation}\label{sec:exp-approx}

The attempt decomposition (Theorem~\ref{thm:pgf}) writes $T = \sum_{i=1}^N A_i + B$ as a geometric sum plus a perturbation.  A classical theorem of~\cite{Renyi56} establishes that normalized geometric sums converge to the exponential distribution; see~\cite{GK96, Kalashnikov97} for comprehensive treatments.  \S\ref{sec:completion} revealed a specific algebraic structure---the coefficient equality $\alpha_N(1) = \alpha_D(1)$---governing the PGF of $T$.  We now show that this structure sharpens R\'enyi's result into:
\begin{enumerate}[(i)]
  \item a quantitative Laplace expansion whose leading error is second-order (the first-order term vanishes as an algebraic identity, not an asymptotic cancellation), and
  \item a sharp two-sided Kolmogorov bound $d_K(T/E[T], \Exp(1)) \asymp p + |\alpha|$ that closes a logarithmic gap left by the smoothing inequality approach.
\end{enumerate}
Both results hold for arbitrary base processes $T_0$ and are driven by the same coefficient equality.

\subsection{The Perturbation Parameter}\label{sec:alpha}

Recall the notation: $\mu = E[T] = (1-p)/(qp)$, $b = E[B]$, and write $W = T/\mu$ for the normalized completion time.  From the decomposition~\eqref{eq:decomp}, $W = U/\mu + B/\mu$ where $U = \sum_{i=1}^N A_i$.

\begin{definition}[Perturbation parameters]\label{def:alpha-beta}
Define
\begin{equation}\label{eq:alpha-beta}
  \alpha = \frac{s(p - qD_1)}{1-p}, \qquad \beta = \frac{1 - qp - qsD_1}{1-p},
\end{equation}
where $D_1 = g'(s)$.
\end{definition}

\begin{proposition}[Properties of the perturbation parameters]\label{prop:alpha-properties}
\mbox{}
\begin{enumerate}[\rm (a)]
  \item \label{lem:identity}\emph{Fundamental identity}: $1 + \alpha - \beta = 0$.
  \item \label{lem:b-mu}\emph{Successful-attempt identity}: $b/\mu = sp/(1-p) - \alpha = qsD_1/(1-p)$.
  \item \label{lem:alpha-vanish}\emph{Vanishing rate}: for fixed $q \in (0,1)$, $|\alpha| = O_q(p|\ln p|)$ as $p \to 0$.
\end{enumerate}
\end{proposition}

\begin{proof}
(a) $1 + \alpha - \beta = 1 + [sp - qsD_1 - 1 + qp + qsD_1]/(1-p) = 1 + (p-1)/(1-p) = 0$, noting that $s + q = 1$.

(b) Differentiating $G_B(w) = g(ws)/p$ at $w = 1$ gives $b = sD_1/p$, so $b/\mu = qsD_1/(1-p)$.  The first equality follows from the definition of $\alpha$.

(c) Set $Y = s^{T_0} \in (0,1]$, so $p = E[Y]$ and $qsD_1 = qE[-Y\ln Y]/|\ln s|$.  Since $\varphi(x) = -x\ln x$ is concave, Jensen gives $E[-Y\ln Y] \leq -p\ln p$, whence $qsD_1 = O_q(p|\ln p|)$.  Both terms in $\alpha = (sp - qsD_1)/(1-p)$ are therefore $O_q(p|\ln p|)$.
\end{proof}

Identity~(a) is the Laplace-domain restatement of the coefficient equality~\eqref{eq:coeff-match}: both arise from $s + q = 1$.  Its consequence is that the first-order error in the exponential approximation vanishes exactly (Theorem~\ref{thm:laplace}).

\begin{remark}[Interpretation of $\alpha$]\label{rem:alpha-interp}
Identity~(b) shows that $\alpha$ measures the deviation of the successful attempt from a ``geometric target'': $\alpha = sp/(1-p) - b/\mu$.  If $\alpha = 0$, the rescaled successful attempt $B/\mu$ has exactly the expected length one would predict from a pure geometric model; if $\alpha \neq 0$, $B$ is longer (when $\alpha < 0$) or shorter (when $\alpha > 0$) than this target.  As we shall see, the Kolmogorov distance decomposes into two independent error sources: a lattice error $\Theta(p)$ from the integrality of the attempt count, and a perturbation error $\Theta(|\alpha|)$ from this mismatch.
\end{remark}

\subsection{The Laplace Expansion}\label{sec:laplace}

\begin{theorem}[Laplace expansion]\label{thm:laplace}
Fix $q \in (0,1)$.  There exist $p_0(q), \epsilon_0(q) > 0$ such that for all $p \leq p_0$ and $\theta > 0$ satisfying $\epsilon := \theta pq/(1-p) \leq \epsilon_0$,
\begin{equation}\label{eq:laplace-expansion}
  E[e^{-\theta W}] - \frac{1}{1+\theta}
  = \frac{\alpha\,\theta^2}{(1+\theta)^2}
  + O\!\left(\frac{|\alpha|^2\,\theta^3}{(1+\theta)^3}\right)
  + O\!\left(\frac{p\,\theta^2}{1+\theta}\right),
\end{equation}
with implicit constants depending only on $q$.
\end{theorem}

\begin{proof}
Set $\epsilon = \theta/\mu = \theta pq/(1-p)$ and $w = e^{-\epsilon}$.  Then $E[e^{-\theta W}] = G_T(w)$.  We expand the numerator $N(w)$ and denominator $D(w)$ of~\eqref{eq:pgf} around $w = 1$.

Using $w = 1 - \epsilon + O(\epsilon^2)$, $u = ws = s - s\epsilon + O(\epsilon^2)$, and $g(u) = p - sD_1\epsilon + O(\epsilon^2)$:

\emph{Numerator.}\; $g(u)(1-u) = (p - sD_1\epsilon + O(\epsilon^2))(q + s\epsilon + O(\epsilon^2)) = pq + (ps - qsD_1)\epsilon + O(\epsilon^2)$.

\emph{Denominator.}\; $1-w+qwg(u) = \epsilon + O(\epsilon^2) + q(1-\epsilon+O(\epsilon^2))(p-sD_1\epsilon+O(\epsilon^2)) = qp + (1-qp-qsD_1)\epsilon + O(\epsilon^2)$.

Dividing both by $pq$ and substituting $\epsilon = \theta pq/(1-p)$:
\begin{equation}\label{eq:ratio}
  G_T(w) = \frac{1 + \alpha\theta + R_N(\theta)}{1 + \beta\theta + R_D(\theta)},
\end{equation}
where $|R_N|, |R_D| = O(\theta^2 p)$ (with implicit constants bounded for fixed $q$, since $s^k D_k \leq \max_{t \geq 1} t^k s^t = (k/(e|\ln s|))^k$). 

By Proposition~\ref{prop:alpha-properties}~\eqref{lem:identity}, $\beta = 1 + \alpha$.  Computing:
\[
  \frac{1 + \alpha\theta}{1 + \beta\theta} - \frac{1}{1+\theta}
  = \frac{(1+\alpha-\beta)\theta + \alpha\theta^2}{(1+\beta\theta)(1+\theta)}
  = \frac{\alpha\theta^2}{(1+\beta\theta)(1+\theta)}.
\]
Since $\beta = 1 + \alpha$: $1/(1+\beta\theta) = 1/((1+\theta)(1 + \alpha\theta/(1+\theta)))$, giving
\[
  \frac{\alpha\theta^2}{(1+\beta\theta)(1+\theta)}
  = \frac{\alpha\theta^2}{(1+\theta)^2} + O\!\left(\frac{|\alpha|^2\theta^3}{(1+\theta)^3}\right).
\]

It remains to bound the contribution of the remainders $R_N, R_D$ to the ratio~\eqref{eq:ratio}.  Writing $G_T(w) - 1/(1+\theta)$ and using $1 + \alpha - \beta = 0$, the numerator becomes $\alpha\theta^2 + (1+\theta)R_N - R_D$.  For the denominator: $1 + \beta\theta + R_D = (1+\theta)(1 + \alpha\theta/(1+\theta) + R_D/(1+\theta))$.  Taking $p_0$ small enough that $|\alpha| \leq 1/4$ and $\epsilon_0 \leq q/4$, the parenthetical factor is at least $1/2$, so $|1+\beta\theta+R_D| \geq (1+\theta)/2$.  The numerator remainder satisfies $|(1+\theta)R_N - R_D| = O((1+\theta)\theta^2 p)$, giving a contribution $O(p\theta^2/(1+\theta))$.
\end{proof}

\begin{remark}[Algebraic origin of the $\theta^1$ cancellation]\label{rem:theta1}
The identity $1 + \alpha - \beta = 0$ is equivalent to $E[W] = 1$ (matching first moments).  Its algebraic root is $s + q = 1$, the same identity underlying the coefficient equality $\alpha_N(1) = \alpha_D(1)$ that drives the DRP.  This is an exact algebraic identity, not an asymptotic cancellation---a structural guarantee that the exponential approximation begins at second order for any base process.
\end{remark}

\begin{corollary}[Exponential limit]\label{cor:exp-limit}
Fix $q \in (0,1)$.  Under memoryless catastrophe with rate $q$, $T/E[T] \xrightarrow{d} \Exp(1)$ as $p = g(s) \to 0$.
\end{corollary}

\begin{proof}
By Proposition~\ref{prop:alpha-properties}~\eqref{lem:alpha-vanish}, $|\alpha| \to 0$.  For each fixed $\theta > 0$, the condition $\epsilon = \theta pq/(1-p) \leq \epsilon_0$ is eventually satisfied, and both error terms in~\eqref{eq:laplace-expansion} vanish.  Hence $E[e^{-\theta W}] \to 1/(1+\theta)$ pointwise.  Pointwise convergence of Laplace transforms on $(0,\infty)$ implies convergence in distribution~\cite{GK96}.
\end{proof}

\begin{remark}[Variable $q$ regime]\label{rem:variable-q}
Corollary~\ref{cor:exp-limit} assumes fixed $q$.  When $q \to 0$ simultaneously with $p \to 0$ (as in the CCP, where $q = m/(n+m) \to 0$), the implicit constants in~\eqref{eq:laplace-expansion} may deteriorate.  The convergence $T/E[T] \xrightarrow{d} \Exp(1)$ in such regimes follows instead from the Bridge Theorem below (Theorem~\ref{thm:bridge}), which requires only $p \to 0$ and uniform control of certain moment ratios of $A$.
\end{remark}


\subsection{The Bridge Theorem}\label{sec:bridge}

The Laplace expansion (Theorem~\ref{thm:laplace}) gives pointwise control on the Laplace transform error.  Converting to Kolmogorov distance $\dK$ via the smoothing inequality (\cite{Kalashnikov97}) yields
\[
  \dK \lesssim \int_0^R \frac{|\alpha|\theta}{(1+\theta)^2}\,d\theta + \frac{1}{R},
\]
which evaluates to $O(|\alpha|\ln(1/|\alpha|))$ upon optimizing $R$.  The logarithmic factor arises because the integrand decays only as $|\alpha|/\theta$ for large $\theta$, accumulating a logarithmic contribution absent from the pointwise bound.

We now develop a direct probabilistic approach that avoids this loss entirely.  The key idea is to decompose the approximation into three steps, each handling one component of $W = U/\mu + B/\mu$.

\begin{lemma}[Failed-attempt tail bound]\label{lem:tail}
Under memoryless catastrophe with rate $q$, the failed-attempt length satisfies $P(A \geq t) \leq s^{t-1}$ for all $t \geq 1$.
\end{lemma}

\begin{proof}
A failure at step $j$ requires survival through steps $1, \ldots, j-1$ (probability $s^{j-1}$), catastrophe at step $j$ (probability $q$), and $T_0 \geq j$.  Hence
\[
  P(A \geq t) = \frac{\sum_{j \geq t} q\,s^{j-1} P(T_0 \geq j)}{1-p}.
\]
Factoring $s^{t-1}$ and using $P(T_0 \geq t+k) \leq P(T_0 \geq k+1)$: the numerator is at most $s^{t-1}\sum_{k \geq 0} q\,s^k P(T_0 \geq k+1) = s^{t-1}(1-p)$.
\end{proof}

\begin{lemma}[Uniform moment-ratio control]\label{lem:moment-ratio}
Fix $q \in (0,1)$.  For any base process $T_0$ and any integer $k \geq 1$,
\[
  E[A^k] \leq M_k(q) := \sum_{t \geq 1} k\,t^{k-1}\,s^{t-1} < \infty,
\]
and $E[A] \geq 1$.  Hence all moment ratios $E[A^k]/E[A]^k$ are bounded by constants depending only on~$q$.
\end{lemma}

\begin{proof}
The tail-sum formula $E[A^k] = \sum_{t \geq 1}[t^k - (t-1)^k]P(A \geq t)$, combined with the mean value theorem bound $t^k - (t-1)^k \leq kt^{k-1}$ and Lemma~\ref{lem:tail}, gives $E[A^k] \leq \sum_{t \geq 1} kt^{k-1}s^{t-1} < \infty$.  The bound $E[A] \geq 1$ is immediate since $A \geq 1$.
\end{proof}

\begin{lemma}[Scaling perturbation]\label{lem:scaling}
For any non-negative random variable $X$ and $r \in (0,1)$,
\[
  \dK(rX, \Exp(1)) \leq \dK(X, \Exp(1)) + (1-r).
\]
\end{lemma}

\begin{proof}
Set $\varepsilon = \dK(X, \Exp(1))$ and $F_E(x) = 1-e^{-x}$.  For all $x \geq 0$:
\[
  |P(rX \leq x) - F_E(x)| \leq \underbrace{|P(X \leq x/r) - F_E(x/r)|}_{\leq\, \varepsilon} + \underbrace{|F_E(x/r) - F_E(x)|}_{=\, e^{-x} - e^{-x/r}}.
\]
The second term is non-negative (since $r < 1$) and maximized at $x^* = r\ln(1/r)/(1-r)$, where it equals $(1-r)r^{r/(1-r)} \leq 1-r$.
\end{proof}

\begin{lemma}[Additive perturbation]\label{lem:additive}
Let $X \geq 0$ and $Y \geq 0$ be independent. Then
\[
  \dK(X+Y, \Exp(1)) \leq \dK(X, \Exp(1)) + E[Y].
\]
\end{lemma}

\begin{proof}
Set $\varepsilon = \dK(X, \Exp(1))$ and $F_E(x) = 1 - e^{-x}$.

\emph{Upper bound.}\; Since $Y \geq 0$, $P(X+Y \leq x) \leq P(X \leq x) \leq F_E(x) + \varepsilon$.

\emph{Lower bound.}\; We must show $F_E(x) - P(X+Y \leq x) \leq E[Y] + \varepsilon$.  Conditioning on~$Y$:
\[
  F_E(x) - P(X+Y \leq x)
  = E\bigl[\bigl(F_E(x)-F_X(x-Y)\bigr)\mathbf 1_{\{Y\le x\}}\bigr]
  + F_E(x)\,P(Y > x).
\]
For the first term, when $Y \leq x$, $F_E(x) - F_X(x-Y)
  \leq \bigl[F_E(x) - F_E(x-Y)\bigr] + \varepsilon
  \leq Y + \varepsilon$, using $|F_E'| \leq 1$.  For the second term, $F_E(x)\,P(Y > x) \leq x P(Y> x) \leq E[Y \cdot \mathbf 1_{\{Y> x\}}]$. Combining them gives
\[
  F_E(x) - P(X+Y \leq x) \leq E[Y] + \varepsilon.  \qedhere
\]
\end{proof}

R\'enyi's classical theorem (\cite{Renyi56}) asserts that normalized geometric sums converge in distribution to $\Exp(1)$.~\cite{Brown90} established the first explicit Kolmogorov-distance bound, showing that the rate is linear in the geometric parameter $p$.

\begin{theorem}[Theorem 2.1 (ii),~\cite{Brown90}]\label{thm:Brown}
Let $U=\sum_{i=1}^{N} X_i$, $N\sim \Geom_0(p)$, where \(0<p<1\), the \(X_i\) are i.i.d.\ nonnegative random variables, \(P(X_1=0)<1\), and \(E[X_1^2]<\infty\). Set $\bar p:=1-p$,
$\rho_2:=\frac{E[X_1^2]}{E[X_1]^2}$.
Then
\begin{equation}\label{eq:Brown-geom-our}
d_K\!\left(\frac{U}{E[U]},\,\Exp(1)\right)
\le
p\max\!\left(\rho_2,\frac{\rho_2}{2\bar p}\right).
\end{equation}
\end{theorem}

We can now state and prove the Bridge Theorem.  The name reflects its role: it bridges the $O(p)$ exponential approximation of the pure geometric sum $U/\mu_U$ to the full completion time $W = T/\mu$.

\begin{theorem}[Bridge Theorem]\label{thm:bridge}
For $0 < p \leq 1/2$,
\begin{equation}\label{eq:bridge}
  \dK(W, \Exp(1)) \leq C_1\,p + 2\,\frac{b}{\mu},
\end{equation}
where $C_1 < \infty$ depends only on the distribution of $A$ (and is uniformly bounded over all base processes $T_0$ for fixed $q$, by Lemma~\ref{lem:moment-ratio}).
\end{theorem}

\begin{proof}
Set $\mu_U = E[U] = (1-p)E[A]/p$ and $r = \mu_U/\mu = 1 - b/\mu \in (0,1)$.

\emph{Step 1: Geometric-convolution bound for $V = U/\mu_U$.}\;
By Lemma~\ref{lem:moment-ratio}, $E[A^k] \leq M_k(q) < \infty$ for all $k \geq 1$, so Brown's moment condition $E[A^2] < \infty$ is satisfied and the moment ratios $\rho_j = E[A^j]/E[A]^j \leq M_j(q)/1 = M_j(q)$ are uniformly bounded.  Theorem~\ref{thm:Brown} applied with \(X_i=A_i\) gives $\dK(V,\Exp(1))\le \rho_2(A)\,p$. Since $\rho_2(A)=E[A^2]/E[A]^2\le M_2(q)$ by Lemma~\ref{lem:moment-ratio},
we may take $C_1=M_2(q)$, and hence 
$\dK(V,\Exp(1))\le C_1 p$.

\emph{Step 2: Scaling perturbation.}\;
Since $U/\mu = rV$ with $r = 1 - b/\mu$, Lemma~\ref{lem:scaling} gives $\dK(U/\mu, \Exp(1)) \leq C_1 p + b/\mu$.

\emph{Step 3: Additive perturbation.}\;
Since $W = U/\mu + B/\mu$ and $U, B$ are independent (by mutual independence of $N$, $(A_i)$, $B$ in~\eqref{eq:decomp}), Lemma~\ref{lem:additive} gives $\dK(W, \Exp(1)) \leq C_1 p + b/\mu + E[B/\mu] = C_1 p + 2b/\mu$. 
\end{proof}

\begin{remark}[Connection to the staircase]\label{rem:bridge-staircase}
The explicit correction terms in the Bridge Theorem---$p = g(s)$ and $b/\mu = qsD_1/(1-p)$---involve only the $1$-jet $\{g(s), g'(s)\}$, matching the staircase level of $\Var(T)$~\eqref{eq:var}. The Brown constant $C_1$ depends on the full distribution of $A$ (via moment ratios), so this staircase connection is more heuristic than exact.
\end{remark}


\subsection{The Sharp Two-Sided Bound}\label{sec:two-sided}

We now combine the Laplace expansion (Theorem~\ref{thm:laplace}) and the Bridge Theorem (Theorem~\ref{thm:bridge}) to obtain a sharp two-sided bound showing that $p + |\alpha|$ is the correct rate.

\begin{theorem}[Two-sided bound]\label{thm:two-sided}
There exist constants $c_A, C_A, \delta_A > 0$ depending only on $q$ such that, for $p + |\alpha| \leq \delta_A$,
\begin{equation}\label{eq:two-sided}
  c_A\,(p + |\alpha|) \leq \dK\!\left(\frac{T}{E[T]},\, \Exp(1)\right) \leq C_A\,(p + |\alpha|).
\end{equation}
\end{theorem}

\begin{proof}[Proof of the upper bound]
From Theorem~\ref{thm:bridge} and Proposition~\ref{prop:alpha-properties}~\eqref{lem:b-mu}: when $\alpha \leq 0$, $b/\mu = |\alpha| + sp/(1-p)$; when $\alpha > 0$, $b/\mu = sp/(1-p) - \alpha \leq sp/(1-p)$.  In both cases, $2b/\mu \leq 2|\alpha| + 2sp/(1-p)$.  Hence
\[
  \dK \leq C_1 p + 2|\alpha| + \frac{2sp}{1-p} \leq C_A(p + |\alpha|). \qedhere
\]
\end{proof}

\begin{proof}[Proof of the lower bound]
We exhibit two independent error sources, each yielding a lower bound.

\emph{Lower bound from discreteness.}\;
Since $T \geq 1$, $W = T/\mu \geq 1/\mu = pq/(1-p)$.  For $x = pq/(2(1-p)) < 1/\mu$: $P(W \leq x) = 0$, while $F_E(x) = 1 - e^{-x} \geq x/2 = pq/(4(1-p))$ (using $1 - e^{-t} \geq t/2$ for $t \in [0,1]$).  Hence
\begin{equation}\label{eq:lower-p}
  \dK \geq \frac{pq}{4(1-p)} = \Omega(p).
\end{equation}

\emph{Lower bound from the Laplace transform.}\;
For any non-negative random variables $W, Z$ and any $\theta > 0$, the integral representation $E[e^{-\theta X}] = \int_0^\infty \theta e^{-\theta x}(1-F_X(x))\,dx$ gives
\begin{equation}\label{eq:laplace-lower}
  \dK(W, Z) \geq |E[e^{-\theta W}] - E[e^{-\theta Z}]|.
\end{equation}
Applying~\eqref{eq:laplace-lower} with $Z \sim \Exp(1)$ and $\theta = 1$: for $p$ sufficiently small, Theorem~\ref{thm:laplace} gives $|E[e^{-W}] - 1/2| = |\alpha|/4 + O(|\alpha|^2) + O(p)$, so
\begin{equation}\label{eq:lower-alpha}
  \dK \geq \frac{|\alpha|}{4} - O(|\alpha|^2) - O(p) \geq c'|\alpha| - c''p
\end{equation}
for suitable $c', c'' > 0$.

\emph{Combining.}\;
Set $K = 2c''/c'$.  Choose $\delta_A$ small enough that both~\eqref{eq:lower-p} and~\eqref{eq:lower-alpha} hold.

\emph{Case 1}: $|\alpha| \leq Kp$.  Then $p + |\alpha| \leq (1+K)p$.  From~\eqref{eq:lower-p}, $\dK \geq qp/4 \geq q(p+|\alpha|)/(4(1+K))$.

\emph{Case 2}: $|\alpha| > Kp$.  Then $p + |\alpha| \leq (1+1/K)|\alpha|$.  From~\eqref{eq:lower-alpha}, $\dK \geq c'|\alpha|/2 \geq c'K(p+|\alpha|)/(2(K+1))$.

In both cases, $\dK \geq c_A(p + |\alpha|)$ with $c_A = \min\{q/(4(1+K)),\; c'K/(2(K+1))\} > 0$.
\end{proof}

The following counterexample shows that the \(p\)-term in the template \(p+|\alpha|\) cannot be dropped.

\begin{proposition}[Counterexample: $p$ cannot be dropped]\label{prop:counterex}
There exists a family of base processes $(T_0^{(\varepsilon)})_{\varepsilon \to 0}$ with $q = 1/2$ such that $|\alpha_\varepsilon| = o(p_\varepsilon)$ but $\dK(W_\varepsilon, \Exp(1)) = \Theta(p_\varepsilon)$.
\end{proposition}

\begin{proof}
Set $q = s = 1/2$ and $P(T_0^{(\varepsilon)} = t) = \varepsilon\,\mathbf{1}_{t=1} + (1-\varepsilon)\,\mathbf{1}_{t = M_\varepsilon}$ where $M_\varepsilon = \lceil 3\log_2(1/\varepsilon)\rceil$.

\emph{Claim 1}: $|\alpha_\varepsilon| = o(p_\varepsilon)$.  Since $2^{-M_\varepsilon} \leq \varepsilon^3$: $p_\varepsilon = \varepsilon/2 + O(\varepsilon^3)$ and $qD_{1,\varepsilon} = \varepsilon/2 + o(\varepsilon) = p_\varepsilon + o(p_\varepsilon)$, so $\alpha_\varepsilon = o(p_\varepsilon)$.

\emph{Claim 2}: $\dK = \Theta(p_\varepsilon)$.  The event $\{T = 1\}$ requires $T_0 = 1$ and no catastrophe: $P(T = 1) = \varepsilon/2$.  At $x_0 = 1/E[T]$: $P(W \leq x_0) = P(T \leq 1) = p_\varepsilon + o(p_\varepsilon)$ while $1-e^{-x_0} = p_\varepsilon/2 + O(p_\varepsilon^2)$.  Hence $\dK \geq p_\varepsilon/2 + o(p_\varepsilon)$.  The upper bound $\dK = O(p_\varepsilon)$ follows from Theorem~\ref{thm:bridge} (since $|\alpha_\varepsilon| = o(p_\varepsilon)$).
\end{proof}

\begin{remark}[Why $p + |\alpha|$ is the correct template]\label{rem:template}
The two-sided bound decomposes the approximation error into two sources with distinct physical origins:
\begin{itemize}
  \item The \emph{$p$-term} reflects geometric discreteness: the attempt count is integer-valued, so $W$ takes values in $(1/\mu)\mathbb{Z}_{>0}$, creating CDF jumps that no continuous distribution can match.
  \item The \emph{$|\alpha|$-term} reflects the successful-attempt perturbation: $T$ is not a pure geometric sum but includes the final attempt $B$, whose expected length may deviate from the rescaling target.
\end{itemize}
The counterexample engineers near-cancellation of the perturbation ($\alpha \approx 0$) while preserving the irreducible lattice error, exposing the $\Theta(p)$ floor.  In the CCP application (\S\ref{sec:ccp-app}), the opposite regime prevails: $p \ll |\alpha|$, so the perturbation term dominates.
\end{remark}

\section{Continuous-Time Theory}\label{sec:continuous}

The discrete theory of \S\S\ref{sec:completion}--\ref{sec:exp-approx} rests on the affine dependence of $G_T$ on the base PGF $g$ at a linearly shifted argument.  In continuous time, the same mechanism operates with PGFs replaced by Laplace transforms and the multiplicative shift $u = ws$ replaced by an additive shift $\sigma = \lambda + r$.  The entire theory carries over, with one notable sharpening: the geometric-tail gap of Theorem~\ref{thm:char} closes, and affine Laplace structure characterizes \emph{full} Poisson resetting with no residual degree of freedom (Theorem~\ref{thm:ct-char}).

\subsection{Setup}\label{sec:ct-setup}

Let $T_0$ be a positive, almost surely finite, continuous random variable with Laplace transform $\hat{f}_0(\lambda) = E[e^{-\lambda T_0}]$.  Catastrophe occurs via a Poisson process of rate $r > 0$, independent of $T_0$.  The attempt decomposition~\eqref{eq:decomp} carries over with $N \sim \Geom_0(p)$ and
\begin{equation}\label{eq:p-ct}
  p = P(T_0 < \tau) = E[e^{-rT_0}] = \hat{f}_0(r),
\end{equation}
where $\tau \sim \Exp(r)$ is the first catastrophe time. Under any age-based mechanism, catastrophe and $T_0$ are independent, so $p = P(T_0 < \tau) = E[S(T_0)]$, the continuous-time analog of~\eqref{eq:p-general}. More generally, for any continuous-time catastrophe mechanism with hazard rate $h: (0,\infty) \to (0,\infty)$, the survival function $S(x) = \exp\bigl(-\int_0^x h(t)\,dt\bigr)$ gives the probability of no catastrophe before time $x$; Poisson resetting is the special case $h \equiv r$, i.e., $S(x) = e^{-rx}$.  We write $\hat{f}_T(\lambda) = E[e^{-\lambda T}]$ for the Laplace transform of the completion time. 

\subsection{The Laplace Transform Formula and Affine Structure}\label{sec:ct-formula}

The following is a standard consequence of the renewal structure under Poisson resetting (see, e.g.,~\cite{EMS20, CS18}).  We state it to exhibit its affine structure, which has not been identified previously.

\begin{theorem}[Laplace transform under Poisson resetting]\label{thm:ct-pgf}
For any base process $T_0$ under Poisson resetting at rate $r$,
\begin{equation}\label{eq:ct-lt}
  \hat{f}_T(\lambda) = \frac{(\lambda + r)\,\hat{f}_0(\lambda + r)}{\lambda + r\,\hat{f}_0(\lambda + r)}, \qquad \sigma := \lambda + r.
\end{equation}
\end{theorem}

\begin{proof}
The successful-attempt Laplace transform is $\hat{f}_B(\lambda) = \hat{f}_0(\sigma)/p$ (tilting by $e^{-rT_0}/p$).  For the failed attempt, $A = \tau \mid \tau < T_0$ with $\tau \sim \Exp(r)$ independent of $T_0$:
\[
  (1-p)\,\hat{f}_A(\lambda) = E[e^{-\lambda\tau}\mathbf{1}_{\tau < T_0}] = \frac{r(1 - \hat{f}_0(\sigma))}{\sigma}.
\]
Composing via $\hat{f}_T = \hat{f}_B \cdot p/[1-(1-p)\hat{f}_A]$, the denominator becomes $1 - r(1-\hat{f}_0(\sigma))/\sigma = (\lambda + r\hat{f}_0(\sigma))/\sigma$, using $\sigma - r = \lambda$, which gives~\eqref{eq:ct-lt}.

\end{proof}

\begin{corollary}\label{cor:ct-ET}
$E[T] = (1-p)/(rp)$, depending on $T_0$ only through $p = \hat{f}_0(r)$.
\end{corollary}

\begin{proof}
By the attempt decomposition, $E[T] = E[\ell]/p$ with $E[\ell] = (1-p)/r$ (the expected attempt length under $\Exp(r)$ catastrophe), giving $E[T] = (1-p)/(rp)$.
\end{proof}

As in the discrete case (\S\ref{sec:affine}), formula~\eqref{eq:ct-lt} is affine in $\hat{f}_0(\sigma)$ with linear shift $\sigma = \lambda + r$:
\begin{equation}\label{eq:ct-affine}
  N(\lambda) = \underbrace{\sigma}_{\alpha_N(\lambda)} \cdot \hat{f}_0(\sigma) + \underbrace{0}_{\beta_N(\lambda)}, \qquad
  D(\lambda) = \underbrace{r}_{\alpha_D(\lambda)} \cdot \hat{f}_0(\sigma) + \underbrace{\lambda}_{\beta_D(\lambda)}.
\end{equation}
At the normalization point $\lambda = 0$: $\beta_N(0) = 0 = \beta_D(0)$, so $\hat{f}_T(0) = 1$ forces
\begin{equation}\label{eq:ct-coeff-match}
  \alpha_N(0) = \alpha_D(0) = r,
\end{equation}
the continuous-time counterpart of~\eqref{eq:coeff-match}.  Since the transform has the same affine-ratio structure with linear shift and coefficient matching at $\lambda=0$, Lemma~\ref{lem:affine-ratio-rec} applies in continuous time as well.

\subsection{Derivative Reduction}\label{sec:ct-drp}

\begin{theorem}[Continuous-time DRP]\label{thm:ct-drp}
Define $D_k = \hat{f}_0^{(k)}(r)$ and $\Delta_k = N^{(k)}(0) - D^{(k)}(0)$.  For all $k \geq 1$,
\begin{equation}\label{eq:ct-Delta}
  \Delta_k =
  \begin{cases}
    -(1-p), & k = 1, \\
    k\,D_{k-1}, & k \geq 2.
  \end{cases}
\end{equation}
In particular, $\Delta_k$ does not contain $D_k$.
\end{theorem}

\begin{proof}
Since $N(\lambda) = (\lambda + r)\hat{f}_0(\lambda + r)$ and $\sigma = \lambda + r$ is linear ($\sigma' = 1$), the Leibniz rule gives $N^{(k)}(0) = rD_k + kD_{k-1}$.  For the denominator: $D^{(1)}(0) = 1 + rD_1$ and $D^{(k)}(0) = rD_k$ for $k \geq 2$.  Hence $\Delta_k = kD_{k-1}$ for $k \geq 2$ (the $D_k$ terms cancel because $\alpha_N(0) = \alpha_D(0) = r$), and $\Delta_1 = (rD_1 + p) - (1 + rD_1) = p - 1$.
\end{proof}

\begin{remark}[Comparison with the discrete DRP]\label{rem:ct-drp-compare}
The discrete formula $\Delta_k = -ks^{k-1}D_{k-1}$ (Theorem~\ref{thm:drp}) contains the factor $s^{k-1}$ from the multiplicative shift $u = ws$ (chain rule coefficient $s$).  The continuous formula $\Delta_k = kD_{k-1}$ has no such factor because the additive shift $\sigma = \lambda + r$ has chain rule coefficient $1$.  The sign difference ($-k$ vs.\ $+k$) reflects opposite monotonicity conventions: $G_T'(1) = E[T] > 0$ while $\hat{f}_T'(0) = -E[T] < 0$.
\end{remark}

\subsection{Characterization: Affine Laplace Structure Is Poisson}\label{sec:ct-char}

In discrete time, affine PGF structure characterizes geometric-tail catastrophe---constant hazard from the second step onward---leaving the first-step hazard $h_1$ as a free parameter (Theorem~\ref{thm:char}).  In continuous time, this residual degree of freedom vanishes.

\begin{definition}[Uniform affine Laplace structure]\label{def:ct-affine}
A continuous-time catastrophe mechanism has \emph{uniform affine Laplace structure} if there exist functions \(a(\lambda),b(\lambda),c(\lambda),d(\lambda)\), analytic in a neighborhood of \(\lambda=0\), depending only on the mechanism and a linear function $\sigma(\lambda) = \lambda + r$ ($r > 0$), such that for all base processes $T_0$,
\[
  \hat{f}_T(\lambda) = \frac{a(\lambda)\,\hat{f}_0(\sigma) + b(\lambda)}{c(\lambda)\,\hat{f}_0(\sigma) + d(\lambda)},
\]
with $a, c$ not identically zero and $c(0)\hat{f}_0(r) + d(0) \neq 0$ for all admissible $\hat{f}_0$.
\end{definition}

\begin{remark}[Why the shift has unit slope]\label{rem:unit-slope}
The restriction $\sigma(\lambda) = \lambda + r$ is not \emph{ad hoc}.  Starting from a general linear shift $\sigma(\lambda) = a\lambda + b$ with $a, b > 0$, the single-point sufficiency condition $p = \hat{f}_0(b)$ for all $T_0$ forces $S(x) = e^{-bx}$ (take $T_0 = \delta_x$), identifying $b$ with the catastrophe rate $r$.  The Poisson catastrophe time then contributes a factor $e^{-rt}$ that combines with the Laplace kernel $e^{-\lambda t}$ into $e^{-(\lambda+r)t}$, giving $a = 1$ without further assumptions.
\end{remark}

For the characterization result below, we temporarily enlarge the admissible
class of base laws: the quantifiers over $T_0$ range over all positive,
almost surely finite random variables, including point masses and two-point
mixtures.  Continuity of $T_0$ is imposed only later, in the lattice-free
sharp-rate statement.

\begin{theorem}[Continuous-time characterization]\label{thm:ct-char}
Among continuous-time age-based catastrophe mechanisms specified by a measurable, locally integrable hazard rate $h: (0,\infty) \to (0,\infty)$ (so that $S(x) = \exp(-\int_0^x h(t)\,dt)$ is well-defined, strictly decreasing, and absolutely continuous), the following are equivalent:
\begin{enumerate}[\rm (a)]
  \item $S(x) = e^{-rx}$ for all $x > 0$ (Poisson resetting; equivalently,
  $h(t) = r$ for a.e.\ $t > 0$).
  \item Uniform affine Laplace structure (Definition~\ref{def:ct-affine}).
  \item Derivative reduction: there exists $r_0 > 0$ such that for every
  positive, almost surely finite $T_0$ and every $k \ge 1$, the $k$-th
  moment of $T$ is determined by
  $\{\hat f_0^{(j)}(r_0)\}_{j=0}^{k-1}$.
  \item Single-point sufficiency: there exists $r_0 > 0$ such that
  $p = \hat f_0(r_0)$ for every positive, almost surely finite $T_0$.
\end{enumerate}
\end{theorem}

\begin{proof}
(a)$\Rightarrow$(b): Theorem~\ref{thm:ct-pgf}.

(b)$\Rightarrow$(c): Suppose
\[
  \hat f_T(\lambda)=\frac{a(\lambda)\hat f_0(\sigma(\lambda))+b(\lambda)}
  {c(\lambda)\hat f_0(\sigma(\lambda))+d(\lambda)},
  \qquad \sigma(\lambda)=\lambda+r_0.
\]
Applying $\hat f_T(0)=1$ to point masses $T_0=\delta_x$ and letting $x\to\infty$ gives $b(0)=d(0)$; back-substitution yields $a(0)=c(0)$. By non-degeneracy, Lemma~\ref{lem:affine-ratio-rec} applies with $R=\hat f_T$, $h=\hat f_0$, $z_*=0$, and $u_0=r_0$. Hence $\hat f_T^{(k)}(0)$ depends only on $\{\hat f_0^{(j)}(r_0)\}_{j=0}^{k-1}$, which is exactly the derivative-reduction property.

(a)$\Rightarrow$(d): By definition~\eqref{eq:p-ct},
\[
  p=P(T_0<\tau)=E[e^{-rT_0}]=\hat f_0(r).
\]

(d)$\Rightarrow$(a): Under (d), for every positive, almost surely finite
$T_0$ we have
\[
  E[S(T_0)] = p = \hat f_0(r_0)=E[e^{-r_0 T_0}].
\]
Taking $T_0=\delta_x$ gives $S(x)=e^{-r_0 x}$ for every $x>0$.

(c)$\Rightarrow$(a): We use only the $k=1$ case.  Thus there exist
$r_0>0$ and a function $\psi:(0,1)\to\mathbb R$ such that
\[
  E[T]=\psi(\hat f_0(r_0))
\]
for every positive, almost surely finite $T_0$.
For a general continuous-time age-based catastrophe mechanism,
\[
  p=E[S(T_0)],\qquad
  E[\ell\,|\,T_0=x]=F(x):=\int_0^x S(t)\,dt,
\qquad
  E[T]=\frac{E[\ell]}{p}
      =\frac{E[F(T_0)]}{E[S(T_0)]}.
\]
Hence
\begin{equation}\label{eq:ct-char-psi}
  \frac{E[F(T_0)]}{E[S(T_0)]}
  =\psi\!\bigl(E[e^{-r_0 T_0}]\bigr)
\end{equation}
for all positive, almost surely finite $T_0$.

\emph{Step 1: point masses determine $\psi$ on $(0,1)$.}
Taking $T_0=\delta_x$ in~\eqref{eq:ct-char-psi} gives
\[
  \psi(e^{-r_0 x})=\frac{F(x)}{S(x)}, \qquad x>0.
\]
Since $x\mapsto e^{-r_0 x}$ is a bijection $(0,\infty)\to(0,1)$, point
masses determine $\psi$ on all of $(0,1)$.

\emph{Step 2: two-point laws force $\psi$ to be M\"obius.}
Take $T_0=\alpha\delta_a+(1-\alpha)\delta_b$ with $0<a<b$.  Writing
\[
  v=\alpha e^{-r_0 a}+(1-\alpha)e^{-r_0 b}\in[e^{-r_0 b},e^{-r_0 a}],
\]
equation~\eqref{eq:ct-char-psi} becomes
\[
  \psi(v)
  =
  \frac{\alpha F(a)+(1-\alpha)F(b)}
       {\alpha S(a)+(1-\alpha)S(b)}.
\]
Solving for $\alpha$ in terms of $v$ shows that the right-hand side is a
linear-fractional function of $v$, namely
\[
  \psi(v)
  =
  \frac{(F(a)-F(b))\,v + F(b)e^{-r_0 a}-F(a)e^{-r_0 b}}
       {(S(a)-S(b))\,v + S(b)e^{-r_0 a}-S(a)e^{-r_0 b}}.
\]
Thus $\psi$ is M\"obius on every interval
$[e^{-r_0 b},e^{-r_0 a}]$.

\emph{Step 3: overlapping intervals glue the local M\"obius maps globally.}
Exactly as in Steps~2--4 of Theorem~\ref{thm:char}, the overlapping-interval
consistency argument shows that all these local M\"obius maps coincide with a
single global one:
\[
  \psi(v)=\frac{Av+B}{Cv+D},\qquad v\in(0,1).
\]
Because $S$ is strictly decreasing, the denominator slope in the two-point
formula is nonzero, so the global denominator coefficient satisfies $C\neq0$.

\emph{Step 4: coefficient consistency forces $F=\rho S + c$.}
Comparing the coefficients of $v$ in the two-point representation with the
global M\"obius map shows that
\[
  \frac{F(a)-F(b)}{S(a)-S(b)}
\]
is independent of the pair $(a,b)$ with $a\neq b$.  Denoting this constant by
$\rho$, we obtain
\[
  F(x)=\rho\,S(x)+c,\qquad x>0
\]
for some constant $c$.

\emph{Step 5: differentiation yields constant hazard everywhere.}
Since $F(x)=\int_0^x S(t)\,dt$ is absolutely continuous and
$F'(x)=S(x)$ a.e., differentiating $F=\rho S+c$ gives
\[
  S(x)=\rho\,S'(x)\qquad\text{for a.e.\ }x>0.
\]
As $S'(x)=-h(x)S(x)$ a.e.\ and $S(x)>0$, we conclude
\[
  h(x)=-\frac1\rho=:r
  \qquad\text{for a.e.\ }x>0.
\]
Hence $S(x)=e^{-rx}$ for all $x>0$, which is exactly (a).
\end{proof}

\begin{remark}[Why the geometric-tail gap closes]\label{rem:gap-closure}
The geometric-tail gap of Theorem~\ref{thm:char} is a joint artifact of two discrete-time limitations: countable test points (requiring M\"obius extension) and integer arithmetic (permitting only differencing).  In continuous time, both limitations vanish simultaneously, forcing full Poisson resetting with no residual degree of freedom.
\end{remark}

\subsection{Exponential Approximation}\label{sec:ct-exp}

The exponential approximation theory of \S\ref{sec:exp-approx} has a continuous-time counterpart.  Define the continuous-time perturbation parameters:
\begin{equation}\label{eq:ct-alpha}
  \alpha = \frac{p + rD_1}{1-p}, \qquad \beta = \frac{1 + rD_1}{1-p},
\end{equation}
where $D_1 = \hat{f}_0'(r)$.  Note that $D_1 = -E[T_0 e^{-rT_0}] < 0$, so the sign conventions in the discrete and continuous perturbation parameters are consistent: both $\alpha$ vanish when $b/\mu$ matches the geometric target (Remark~\ref{rem:alpha-interp}). 
The fundamental identity 
\begin{equation}\label{lem:ct-identity}
    1 + \alpha - \beta = 0 
\end{equation} holds by the same one-line computation as Proposition~\ref{prop:alpha-properties}\eqref{lem:identity}, replacing $s+q=1$ by direct cancellation of $rD_1$. Differentiating $\hat{f}_B(\lambda) = \hat{f}_0(\lambda+r)/p$ at $\lambda = 0$ gives $b = -D_1/p$, so
\begin{equation}\label{lem:ct-b-mu}
  \frac{b}{\mu} = \frac{-rD_1}{1-p} = \frac{p}{1-p} - \alpha,
\end{equation}
the continuous-time analog of Proposition~\ref{prop:alpha-properties}\eqref{lem:b-mu}.

\begin{theorem}[Continuous-time Laplace expansion]\label{thm:ct-laplace}
Under Poisson resetting at rate $r$, let $W = T/E[T]$.  For $\epsilon := \theta rp/(1-p) \leq \epsilon_0$,
\begin{equation}\label{eq:ct-laplace}
  E[e^{-\theta W}] - \frac{1}{1+\theta}
  = \frac{\alpha\,\theta^2}{(1+\theta)^2}
  + O\!\left(\frac{|\alpha|^2\,\theta^3}{(1+\theta)^3}\right)
  + O\!\left(\frac{p\,\theta^2}{1+\theta}\right),
\end{equation}
with implicit constants depending on $|D_1|$ and $|D_2|$, which satisfy $|D_k| = |\hat{f}_0^{(k)}(r)| = E[T_0^k e^{-rT_0}] \leq (k/(er))^k$.
\end{theorem}

\begin{proof}
Set $\mu = (1-p)/(rp)$, $\epsilon = \theta/\mu$, and expand $\hat{f}_0(r+\epsilon) = p + D_1\epsilon + O(\epsilon^2)$.  Then $N(\epsilon) = (r+\epsilon)(p + D_1\epsilon + O(\epsilon^2)) = rp + (p+rD_1)\epsilon + O(\epsilon^2)$ and $D(\epsilon) = \epsilon + r(p + D_1\epsilon + O(\epsilon^2)) = rp + (1+rD_1)\epsilon + O(\epsilon^2)$.  Dividing by $rp$ and substituting $\epsilon = \theta rp/(1-p)$ yields the ratio $(1+\alpha\theta+R_N)/(1+\beta\theta+R_D)$ with $|R_N|, |R_D| = O(\theta^2 p)$.  The identity $1+\alpha-\beta = 0$ ensures the $\theta^1$-coefficient vanishes.  The remainder analysis is identical to the proof of Theorem~\ref{thm:laplace}.
\end{proof}

\begin{lemma}[Continuous-time $\alpha \to 0$]\label{lem:ct-alpha}
For fixed $r > 0$, $|\alpha| = O(p|\ln p|)$ as $p = \hat{f}_0(r) \to 0$.
\end{lemma}

\begin{proof}
Set $Y = e^{-rT_0} \in (0,1]$, so $p = E[Y]$ and $rD_1 = E[Y\ln Y]$ (since $Y\ln Y = -rT_0 e^{-rT_0}$), giving $-rD_1 = E[-Y\ln Y]$.  Since $\varphi(x) = -x\ln x$ is concave, Jensen's inequality gives $E[-Y\ln Y] \leq -p\ln p$.  Hence $|p + rD_1| \leq p + E[-Y\ln Y] \leq p(1 + |\ln p|)$, so $|\alpha| = O(p|\ln p|)$.
\end{proof}

In continuous time, the lattice error that contributes the $\Theta(p)$ floor in the discrete two-sided bound (Theorem~\ref{thm:two-sided}) disappears: when $T_0$ has a continuous distribution, so does $T$, and the CDF of $W$ has no jumps.  The sharp rate therefore reduces to $|\alpha|$ alone.  The following lemma provides the necessary moment convergence for the Bridge Theorem.

\begin{lemma}[Escape to infinity]\label{lem:escape}
Fix $r > 0$.  Let $(T_{0,n})$ be a sequence of positive random variables with $p_n = E[e^{-rT_{0,n}}] \to 0$.  Then $T_{0,n} \to \infty$ in probability.  Consequently, for each fixed $t \geq 0$, $P(A_n \geq t) \to e^{-rt}$, and for each fixed integer $k \geq 1$, $E[A_n^k] \to k!/r^k$.
\end{lemma}

\begin{proof}
For any fixed $x > 0$, $p_n \geq e^{-rx}P(T_{0,n} \leq x)$, so $P(T_{0,n} \leq x) \leq e^{rx}p_n \to 0$.

For the failed-attempt tail: since $\tau \sim \Exp(r)$ is independent of $T_{0,n}$,
\[
  P(A_n \geq t) = \frac{P(\tau \geq t,\, \tau < T_{0,n})}{1-p_n}
  = \frac{E[(e^{-rt} - e^{-rT_{0,n}})\mathbf{1}_{\{T_{0,n} > t\}}]}{1-p_n}.
\]
As $n \to \infty$: $E[e^{-rT_{0,n}}\mathbf{1}_{\{T_{0,n} > t\}}] \leq p_n \to 0$ and $P(T_{0,n} > t) \to 1$, so the numerator tends to $e^{-rt}$ and the denominator to $1$.

The uniform bound $P(A_n \geq t) \leq e^{-rt}/(1-p_n) \leq 2e^{-rt}$ (for $p_n \leq 1/2$) provides a dominator for $t^{k-1}P(A_n \geq t)$, so dominated convergence gives $E[A_n^k] = k\int_0^\infty t^{k-1}P(A_n \geq t)\,dt \to k!/r^k$.
\end{proof}

\begin{theorem}[Continuous-time sharp bound]\label{thm:ct-sharp}
Under Poisson resetting at rate $r$, for $T_0$ with continuous distribution, 
and $|\alpha|/p \to \infty$ as $p \to 0$:
\begin{equation}\label{eq:ct-sharp}
  \dK\!\left(\frac{T}{E[T]},\, \Exp(1)\right) \in \bigl[1-o(1),\; 2+o(1)\bigr] \cdot |\alpha|.
\end{equation}
\end{theorem}

\begin{proof}

\emph{Upper bound.}\;
The Bridge Theorem (Theorem~\ref{thm:bridge}) applies in continuous time: its three steps---Brown's bound, scaling, and additive perturbation---require only that $N \sim \Geom_0(p)$, the $A_i$ be i.i.d.\ positive with $E[A^2] < \infty$, and $B$ be independent of $U$; none assumes integrality.  The tail bound $P(A \geq t) \leq e^{-rt}/(1-p)$ from the proof of Lemma~\ref{lem:escape} gives $E[A^k] < \infty$ for all $k$, verifying Brown's moment condition.  By Lemma~\ref{lem:escape}, $\rho_j = E[A^j]/E[A]^j \to j!$ as $p \to 0$, so the Brown constant $C_1 = O(1)$.  By~\eqref{lem:ct-b-mu}, $2b/\mu \leq 2|\alpha| + 2p/(1-p)$.  Hence $\dK \leq C_1 p + 2|\alpha| + 2p/(1-p) = (2+o(1))|\alpha|$, using $p = o(|\alpha|)$.

\emph{Lower bound.}\;
Apply the Laplace expansion (Theorem~\ref{thm:ct-laplace}) with $\theta = \theta_n = (|\alpha|/(2p))^{1/2} \to \infty$.  The validity condition $\epsilon = \theta rp/(1-p) \to 0$ holds since $\epsilon = O(r\sqrt{p|\alpha|}) \to 0$ by Lemma~\ref{lem:ct-alpha}.  The error terms satisfy $p\theta^2/(1+\theta) \leq p\theta = \sqrt{p|\alpha|/2} = o(|\alpha|)$ and $|\alpha|^2\theta^3/(1+\theta)^3 \leq |\alpha|^2 = o(|\alpha|)$.  The leading term: $|\alpha|\theta^2/(1+\theta)^2 = (1-o(1))|\alpha|$ since $\theta \to \infty$.  By~\eqref{eq:laplace-lower}, $\dK \geq (1-o(1))|\alpha|$.
\end{proof}

\begin{remark}[Disappearance of the lattice error]\label{rem:no-lattice}
In the discrete two-sided bound (Theorem~\ref{thm:two-sided}), the template is $p + |\alpha|$ because $W$ takes values in $(1/\mu)\mathbb{Z}_{>0}$, creating an irreducible $\Omega(p)$ lattice error (Proposition~\ref{prop:counterex}).  In continuous time, this error vanishes.  When $|\alpha|/p \to \infty$, the dominant error source is the successful-attempt perturbation encoded by $\alpha$; when $\alpha \approx 0$, a residual $O(p)$ error from higher-order terms persists.
\end{remark}


\subsection{Worked Example: Brownian First Passage}\label{sec:brownian}

To illustrate the continuous-time theory concretely, consider a Brownian particle starting at $x_0 > 0$ with an absorbing barrier at the origin, subject to Poisson resetting at rate $r$.  This is the original model of~\cite{EM11}.  The base first-passage time has Laplace transform $\hat{f}_0(\lambda) = e^{-x_0\sqrt{2\lambda}}$, giving
\[
  p = e^{-x_0\sqrt{2r}}, \qquad
  E[T] = \frac{e^{x_0\sqrt{2r}} - 1}{r}, \qquad
  D_1 = -\frac{x_0}{\sqrt{2r}}\,e^{-x_0\sqrt{2r}}.
\]
The perturbation parameter:
\begin{equation}\label{eq:brownian-alpha}
  \alpha = \frac{(1 - x_0\sqrt{r/2})\,e^{-x_0\sqrt{2r}}}{1 - e^{-x_0\sqrt{2r}}}.
\end{equation}
As $x_0\sqrt{r} \to \infty$: $|\alpha| \sim x_0\sqrt{r/2}\cdot p$, so $|\alpha|/p \to \infty$---the perturbation term dominates, and Theorem~\ref{thm:ct-sharp} gives $\dK \in [1-o(1), 2+o(1)] \cdot |\alpha|$.  

At the critical point $x_0\sqrt{r/2} = 1$, the perturbation parameter $\alpha$ vanishes exactly: the leading $\alpha\theta^2/(1+\theta)^2$ term in~\eqref{eq:ct-laplace} disappears, leaving only higher-order contributions.  This does not imply distributional proximity to $\Exp(1)$: at the critical point $p = e^{-2}$ is a fixed constant (not tending to zero), so the exponential limit theorem does not apply.


\subsection{Discrete--Continuous Correspondence}\label{sec:correspondence}

Table~\ref{tab:correspondence} summarizes the parallel between the discrete and continuous theories.  Every structural feature---affine form, coefficient matching, DRP, perturbation identity, Laplace expansion, sharp rate---carries over.  The sole qualitative difference lies in the characterization's scope and the sharp-rate template.

\begin{table}[ht]
\centering
\small
\renewcommand{\arraystretch}{1.25}
\begin{tabular}{lll}
\toprule
& \textbf{Discrete time} & \textbf{Continuous time} \\
\midrule
\emph{Transform} & PGF $G_T(w) = E[w^T]$ & Laplace $\hat{f}_T(\lambda) = E[e^{-\lambda T}]$ \\
\emph{Normalization point} & $w = 1$ & $\lambda = 0$ \\
\emph{Shift} & $u = ws$ (multiplicative) & $\sigma = \lambda + r$ (additive) \\
\emph{Catastrophe parameter} & $q \in (0,1)$ & $r > 0$ \\
\emph{Success probability} & $p = g(s)$ & $p = \hat{f}_0(r)$ \\
\emph{$E[T]$} & $(1-p)/(qp)$ & $(1-p)/(rp)$ \\
\emph{Coefficient match} & $\alpha_N(1) = \alpha_D(1) = q$ & $\alpha_N(0) = \alpha_D(0) = r$ \\
\emph{DRP ($k \geq 2$)} & $\Delta_k = -ks^{k-1}D_{k-1}$ & $\Delta_k = kD_{k-1}$ \\
\emph{Perturbation $\alpha$} & $s(p-qD_1)/(1-p)$ & $(p+rD_1)/(1-p)$ \\
\emph{Fundamental identity} & $1 + \alpha - \beta = 0$ & $1 + \alpha - \beta = 0$ \\
\emph{Laplace expansion} & $\alpha\theta^2/(1+\theta)^2 + \text{h.o.t.}$ & $\alpha\theta^2/(1+\theta)^2 + \text{h.o.t.}$ \\
\emph{Bridge Theorem} & $\dK \leq C_1 p + 2b/\mu$ & $\dK \leq C_1 p + 2b/\mu$ \\
\emph{Sharp-rate template} & $p + |\alpha|$ & $|\alpha|$ (no lattice error) \\
\emph{Characterizes} & geometric tail & Poisson (full memoryless) \\
\emph{Gap} & $h_1$ free & none \\
\bottomrule
\end{tabular}
\caption{Discrete--continuous correspondence.  The two theories share every structural feature.  The characterization sharpens in continuous time (the geometric-tail gap closes, Remark~\ref{rem:gap-closure}), and the sharp-rate template simplifies (the lattice error vanishes for continuous $T_0$, Remark~\ref{rem:no-lattice}).}\label{tab:correspondence}
\end{table}

\section{Application: Coupon Collector with Reset}\label{sec:ccp-app}

The general theory of \S\S\ref{sec:completion}--\ref{sec:exp-approx} applies to any base process $T_0$.  We now instantiate it to the coupon collector problem with reset, obtaining closed-form moments, a sharp Kolmogorov bound that closes the logarithmic gap, and a discontinuous phase transition in the limit law.  The pattern throughout is: compute CCP-specific quantities ($p$, $D_1$, $D_2$, $\alpha$), then substitute into general formulas.  The sole exception is the lower bound in the sharp Kolmogorov estimate (\S\ref{sec:ccp-sharp}), which requires a product-structure argument specific to the CCP.


\subsection{Model}\label{sec:ccp-model}

A deck contains $n + m$ coupon types: types $1, \ldots, n$ are \emph{standard}, and types $n+1, \ldots, n+m$ are \emph{reset}.  At each step, one type is drawn uniformly at random with replacement.  Standard draws accumulate; any reset draw erases all collected coupons.  The game ends when all $n$ standard types have been collected.  This is the framework of \S\ref{sec:model} with $T_0$ the classical coupon collector completion time and
\[
  q = \frac{m}{n+m}, \qquad s = \frac{n}{n+m}.
\]
A closely related model was studied by~\cite{JT24}, who derived expected values via Markov chain methods; our contribution is the distributional theory.

\textbf{Notation.}\; $C = \binom{n+m}{n}$;\; $H = \sum_{j=m+1}^{n+m} 1/j = H_{n+m} - H_m$;\; $H^{(r)} = \sum_{j=m+1}^{n+m} 1/j^r = H^{(r)}_{n+m} - H^{(r)}_m$.


\subsection{Specialization of the PGF Data}\label{sec:ccp-pgf}

The classical CCP has $T_0 = X_1 + \cdots + X_n$ with $X_k \sim \Geom_1((n{-}k{+}1)/n)$ independent, giving
\begin{equation}\label{eq:ccp-pgf}
  g(z) = \prod_{k=1}^{n} \frac{(n{-}k{+}1)z/n}{1 - (k{-}1)z/n}.
\end{equation}
The general theory requires three inputs: $p = g(s)$, $D_1 = g'(s)$, and $D_2 = g''(s)$.

\begin{lemma}[Success probability]\label{lem:ccp-p}
$p = g(s) = 1/C = 1/\binom{n+m}{n}$.
\end{lemma}

\begin{proof}
\emph{Algebraic.}\; At $z = s = n/(n+m)$, each factor of~\eqref{eq:ccp-pgf} evaluates to $(n{-}k{+}1)/(n{+}m{-}k{+}1)$.  The product telescopes: $\prod_{k=1}^n (n{-}k{+}1)/(n{+}m{-}k{+}1) = n!\,m!/(n{+}m)! = 1/C$.

\emph{Combinatorial.}\; An attempt succeeds if and only if all $n$ standard types appear before any reset type in the first-appearance permutation.  By symmetry of i.i.d.\ uniform draws, this probability is $n!\,m!/(n+m)! = 1/C$.
\end{proof}

The derivatives are computed by logarithmic differentiation, exploiting the product structure $\log g = \sum_{k=1}^n \log G_{X_k}$.

\begin{lemma}[PGF derivatives at $z = s$]\label{lem:ccp-derivs}
\begin{equation}\label{eq:ccp-D1}
  D_1 = g'(s) = \frac{(n+m)^2\,H}{n\,C},
\end{equation}
\begin{equation}\label{eq:ccp-D2}
  D_2 = g''(s) = \frac{(n+m)^3}{n^2\,C}\bigl((n+m)(H^2 + H^{(2)}) - 2H\bigr).
\end{equation}
More generally, the $r$-th logarithmic derivative of $g$ at $z = s$ is a polynomial in $H, H^{(2)}, \ldots, H^{(r)}$ with rational coefficients in $(n, m)$.
\end{lemma}

\begin{proof}[Proof sketch]
Write $r_k(z) = G'_{X_k}(z)/G_{X_k}(z) = 1/(z\varphi_k(z))$ with $\varphi_k(z) = 1-(k{-}1)z/n$.  At $z = s$: $\varphi_k(s) = (n{+}m{-}k{+}1)/(n{+}m)$, so $r_k(s) = (n{+}m)^2/(n(n{+}m{-}k{+}1))$.  Setting $j = n{+}m{-}k{+}1$:
\[
  (\log g)'(s) = \sum_{k=1}^n r_k(s) = \frac{(n+m)^2}{n}\sum_{j=m+1}^{n+m}\frac{1}{j} = \frac{(n+m)^2\,H}{n}.
\]
Then $D_1 = g(s)\cdot(\log g)'(s)$.  The second derivative follows from $g'' = g[(\log g)'' + ((\log g)')^2]$ and an analogous computation of $(\log g)''(s)$, which introduces $H^{(2)}$ via the partial fractions of $r'_k(s)$.
\end{proof}

The pattern is transparent: the $r$-th logarithmic derivative introduces $H^{(r)}$ because each differentiation of $r_k(z) = 1/(z\varphi_k(z))$ adds one power of $1/\varphi_k(s) = (n{+}m)/j$, and summing over $k$ produces $\sum j^{-r}$.


\subsection{Moments}\label{sec:ccp-moments}

Substituting the CCP data $(p,q,D_1)$ into the general formulas of \S\ref{sec:drp} gives closed-form first two moments:
\begin{equation}\label{eq:ccp-ET}
  E[T] = \frac{n+m}{m}(C-1), \qquad
  \Var(T) = \frac{n+m}{m^2}\bigl((C-1)((n{+}m)C+n) - 2m(n{+}m)CH\bigr).
\end{equation}
For $m = 1$: $E[T] = n(n+1)$.  For fixed $m$ as $n \to \infty$: $E[T] \sim n^{m+1}/(m \cdot m!)$.  The variance involves both $p = 1/C$ and $D_1$---but not $D_2$---exactly as the DRP guarantees.  Higher moments follow analogously from the recursion~\eqref{eq:recursion}. In particular, the third centered moment is the first level at which \(D_2\), and hence \(H^{(2)}\), enters; more generally, the \(k\)-th centered moment first involves the new harmonic statistic \(H^{(k-1)}\).

The staircase of \S\ref{sec:staircase} takes a particularly illuminating form:
\begin{center}
\begin{tabular}{lccc}
\toprule
\textbf{Moment of $T$} & \textbf{New input (DRP)} & \textbf{CCP realization} & \textbf{Origin} \\
\midrule
$E[T]$ & $g(s)$ & $C = \binom{n+m}{n}$ & product telescope \\[2pt]
$\Var(T)$ & $+ g'(s)$ & $+ H = \sum 1/j$ & 1st log-derivative \\[2pt]
$\mu_3(T)$ & $+ g''(s)$ & $+ H^{(2)} = \sum 1/j^2$ & 2nd log-derivative \\[2pt]
$\mu_k(T)$ & $+ g^{(k-1)}(s)$ & $+ H^{(k-1)} = \sum 1/j^{k-1}$ & $(k{-}1)$-th log-deriv. \\
\bottomrule
\end{tabular}
\end{center}
Each row adds exactly one new finite harmonic sum: $H^{(k-1)}=\sum_{j=m+1}^{n+m} j^{-(k-1)}$. The CCP thus provides the sharpest illustration of the harmonic complexity staircase---a feature absent in base processes with rational PGF structure (\S\ref{sec:multiphase}).


\subsection{Phase Transition and Sharp Kolmogorov Bound}\label{sec:ccp-sharp}

The CCP perturbation parameter, obtained by substituting $p = 1/C$ and $qD_1 = m(n+m)H/(nC)$ into Definition~\ref{def:alpha-beta}, is
\begin{equation}\label{eq:ccp-alpha}
  \alpha = \frac{s - mH}{C-1}, \qquad |\alpha| \sim \frac{m!\,m\ln n}{n^m} \quad\text{as } n \to \infty.
\end{equation}
Since $p = 1/C \sim m!/n^m$ and $|\alpha| \sim m!\,m\ln n/n^m$, the ratio $p/|\alpha| \sim 1/(m\ln n) \to 0$: the perturbation term dominates.

\medskip
\noindent\textbf{The Gumbel-to-Exponential phase transition.}\;
For the classical CCP ($m = 0$), the limit law is Gumbel: $(T_0 - n\ln n)/n \xrightarrow{d} \text{Gumbel}$ (\cite{ER61, BB65}).  For $m \geq 1$, the general theory (applied via the Bridge Theorem, Theorem~\ref{thm:bridge}) gives $T/E[T] \xrightarrow{d} \Exp(1)$.  The transition at $m = 0 \to m \geq 1$ is discontinuous:

\begin{center}
\small
\renewcommand{\arraystretch}{1.2}
\begin{tabular}{lll}
\toprule
& \textbf{Classical ($m = 0$)} & \textbf{Reset ($m \geq 1$)} \\
\midrule
Scale & $\Theta(n\log n)$ & \(\sim n^{m+1}/(m\cdot m!)\) \\
Fluctuations & $\Theta(n) \ll E[T]$ & $\Theta(E[T])$ \\
Coefficient of variation & $\to 0$ & $\to 1$ \\
Limit law & Gumbel & Exponential \\
Mechanism & last coupon (extreme value) & lucky streak (geometric) \\
\bottomrule
\end{tabular}
\end{center}

In the classical CCP, the bottleneck is the last missing coupon---an extreme value problem yielding the Gumbel law.  A single reset coupon destroys this mechanism: completion now requires a lucky streak (an attempt succeeding without catastrophe), and the number of attempts is geometric.  The geometric waiting time's memorylessness propagates to $T/E[T] \to \Exp(1)$, making the completion time fundamentally unpredictable.

\begin{remark}[CCP convergence via the Bridge Theorem]\label{rem:ccp-bridge}
The CCP has $q = m/(n+m) \to 0$ as $n \to \infty$, so Corollary~\ref{cor:exp-limit} (which assumes fixed $q$) does not apply directly.  Instead, CCP convergence follows from the Bridge Theorem (Theorem~\ref{thm:bridge}): the conditions $p \to 0$ and uniform boundedness of moment ratios of $A$ are verified below in the proof of Theorem~\ref{thm:ccp-sharp}.
\end{remark}

\begin{theorem}[CCP sharp bound]\label{thm:ccp-sharp}
For the coupon collector with $n$ standard and $m \geq 1$ reset coupons, $m$ fixed, $n \to \infty$:
\begin{equation}\label{eq:ccp-sharp}
  \dK\!\left(\frac{T}{E[T]},\, \Exp(1)\right) \in \bigl[1-o(1),\; 2+o(1)\bigr] \cdot |\alpha|.
\end{equation}
\end{theorem}

\begin{proof}[Proof of the upper bound]
To apply the Bridge Theorem (Theorem~\ref{thm:bridge}), it suffices to show
that the moment ratios
$\rho_j:=\frac{E[A^j]}{E[A]^j}, j=2,3$, are bounded uniformly in $n$ for each fixed $m\ge1$. We first bound the numerator. By Lemma~\ref{lem:tail},
$P(A\ge t)\le s^{t-1}, t\ge1$.
Hence, by the tail-sum formula,
\[
E[A^k]
=
\sum_{t\ge1}\bigl(t^k-(t-1)^k\bigr)P(A\ge t)
\le
k\sum_{t\ge1} t^{k-1}s^{t-1}.
\]
Since $s=1-q\le e^{-q}$,
\[
\sum_{t\ge1} t^{k-1}s^{t-1}
\le
2\sum_{t\ge1} t^{k-1}e^{-qt}
\le
C_k q^{-k},
\]
for a constant $C_k$ depending only on $k$. Therefore, $E[A^k]\le C_k q^{-k}$.

For the lower bound on $E[A]$, use the explicit failed-attempt law
$P(A=j)=\frac{q\,s^{j-1}P(T_0>j)}{1-p}, j\ge1$.
In the CCP, $T_0\ge n$ almost surely, so $P(T_0>j)=1$ for $1\le j\le n-1$.
Hence
\[
E[A]
=
\frac{1}{1-p}\sum_{j\ge1} j q s^{j-1}P(T_0>j)
\ge
\sum_{j=1}^{n-1} j q s^{j-1}
=
\frac{1-ns^{\,n-1}+(n-1)s^n}{q}.
\]
Since $q=m/(n+m)$ and $s=n/(n+m)$, for each fixed $m\ge1$,
$1-ns^{\,n-1}+(n-1)s^n \to 1-(m+1)e^{-m}>0$,
so there exists $c_m>0$ such that
$E[A]\ge \frac{c_m}{q}
$ for all sufficiently large $n$.

Combining the two estimates gives
\[
\frac{E[A^k]}{E[A]^k}
\le
\frac{C_k q^{-k}}{(c_m q^{-1})^k}
=
\frac{C_k}{c_m^k}
=
O_m(1).
\]
Hence $\rho_2,\rho_3=O_m(1)$, and Brown's constant in
Theorem~\ref{thm:bridge} satisfies $C_1=O_m(1)$. Thus, from~\eqref{eq:ccp-alpha} and Proposition~\ref{prop:alpha-properties}~\eqref{lem:b-mu}: $b/\mu = \frac{mH}{C-1} = |\alpha| + \frac{s}{C-1}$, and $\frac{s}{C-1} = O(p)$. Finally, $\dK \leq O_m(p) + 2|\alpha| + O(p) = (2+o(1))|\alpha|$.
\end{proof}

\begin{proof}[Proof of the lower bound]
This is the one point where CCP-specific structure goes beyond substitution into general formulas.  Choose $\theta = \theta_n = (\ln n)^{1/2}$ and set $w = e^{-\theta/\mu}$, $u = ws$, $\epsilon = \theta/\mu$.  By~\eqref{eq:laplace-lower}, $\dK \geq |G_T(w) - 1/(1+\theta)|$.

\emph{Step 1: PGF reduction.}\;
From~\eqref{eq:pgf}, setting $c = g(u)/p$:
\[
  G_T(w) = \frac{c + O_m(\theta p)}{c + \theta + O_m(\theta p)}.
\]
Hence $G_T(w) - 1/(1+\theta) = (c-1)\theta/((c+\theta)(1+\theta)) + O_m(\theta p)$.

\emph{Step 2: Product-structure evaluation of $\log c$.}\;
The product structure $g = \prod_{k=1}^n G_{X_k}$ gives 
\[\log c = \sum_{k=1}^n \log(G_{X_k}(u)/G_{X_k}(s)).\]  
Each ratio $G_{X_k}(u)/G_{X_k}(s)$ expands as 
\[
  \frac{G_{X_k}(u)}{G_{X_k}(s)}
  =
  \frac{w(1-a_k)}{1-a_k w}=\frac{1-\delta}{1+x_k},
\]
where $x_k = a_k\delta/(1-a_k)$ with $a_k = (k{-}1)/(n{+}m)$ and $\delta = 1 - w$.  Since $\sup_k x_k = O_m(\sqrt{\ln n}/n^m) \to 0$, the expansion $\log(1+x_k) = x_k + O(x_k^2)$ is uniformly valid.  Summing and using the identity $\sum_{k=1}^n a_k/(1-a_k) = (n+m)H - n$:
\[
  \log c = -\delta(n+m)H + O(\delta^2 S_2),
\]
where $S_2 = \sum_{k=2}^n a_k^2/(1-a_k)^2 = O_m(n^2)$.  Converting $\delta(n+m)H$ to $\theta|\alpha|(1+o(1))$ and verifying $\delta^2 S_2 = o(\theta|\alpha|)$:
\[
  \log c = -\theta|\alpha|(1+o(1)).
\]
Since $\theta|\alpha| \to 0$, this gives $c = 1 - \theta|\alpha|(1+o(1))$.

\emph{Step 3: Assembly.}\;
$(c-1)\theta/((c+\theta)(1+\theta)) = -(1+o(1))\theta^2|\alpha|/(1+\theta)^2$, and the correction $O_m(\theta p) = o(|\alpha|)$.  Hence $\dK \geq (1+o(1))\theta^2|\alpha|/(1+\theta)^2 = (1-o(1))|\alpha|$, using $\theta^2/(1+\theta)^2 = 1 - O(1/\sqrt{\ln n})$.
\end{proof}

\section{Application: Multi-Phase Task with Catastrophe}\label{sec:multiphase}

The CCP of \S\ref{sec:ccp-app} features a base PGF that is a product of $n$ distinct M\"obius factors, producing a \emph{harmonic complexity staircase} whose successive levels introduce new finite harmonic sums $H^{(r)}$.  We now present a structurally contrasting application: a base process whose PGF is a \emph{power} of a single M\"obius factor, yielding a purely \emph{algebraic} staircase.  The DRP applies identically in both cases; the difference lies entirely in what it saves.

\subsection{Model and the Algebraic Staircase}\label{sec:mp-model}

A task consists of $r \geq 2$ sequential phases, each requiring $\Geom_1(\pi)$ steps independently.  The base completion time is $T_0 = X_1 + \cdots + X_r$, $X_i \stackrel{\text{i.i.d.}}{\sim} \Geom_1(\pi)$, with PGF
\begin{equation}\label{eq:mp-pgf}
  g(z) = \left[\frac{\pi z}{1-(1-\pi)z}\right]^{\!r}.
\end{equation}
Catastrophe at each step (probability $q$, independent) destroys all accumulated progress, resetting the phase counter to zero.  Define $\sigma = 1-(1-\pi)s = \pi + q(1-\pi)$ and $\rho = \pi s/\sigma$.  Then
\begin{equation}\label{eq:mp-p}
  p = g(s) = \rho^r, \qquad E[T] = \frac{1-\rho^r}{q\rho^r}, \qquad D_1 = g'(s) = \frac{r\rho^r}{s\sigma},
\end{equation}
and the variance, obtained from~\eqref{eq:var}, depends on $p$ and $D_1$ but not on $D_2$---exactly as the DRP guarantees.

The key structural feature is that $\log g = r\log f$ with $f(z) = \pi z/(1-(1-\pi)z)$, so every logarithmic derivative
\begin{equation}\label{eq:mp-logderiv}
  (\log g)^{(k)}(s) = r(k-1)!\left[\frac{(-1)^{k-1}}{s^k} + \frac{(1-\pi)^k}{\sigma^k}\right]
\end{equation}
is a \emph{rational function} of $(s, \pi, r)$---no new generalized harmonic statistic appears at any level.  Displayed alongside the CCP realization:

\begin{center}
\small
\renewcommand{\arraystretch}{1.2}
\begin{tabular}{lccc}
\toprule
\textbf{Moment of $T$} & \textbf{New input (DRP)} & \textbf{CCP realization} & \textbf{Multi-phase realization} \\
\midrule
$E[T]$ & $g(s)$ & $C = \binom{n+m}{n}$ & $\rho^r$ \\[2pt]
$\Var(T)$ & $+ g'(s)$ & $+ H = \sum 1/j$ & $+ r/(s\sigma)$ \\[2pt]
$\mu_3(T)$ & $+ g''(s)$ & $+ H^{(2)} = \sum 1/j^2$ & $+\, r\cdot[\text{rational in }s,\sigma,\pi]$ \\[2pt]
$\mu_k(T)$ & $+ g^{(k-1)}(s)$ & $+ H^{(k-1)}$ & $+ r \cdot [\text{rational}]$ \\
\bottomrule
\end{tabular}
\end{center}

The contrast traces to the PGF's analytic structure: the CCP's $n$ distinct poles produce new power sums $\sum j^{-k}$ at each level, while the multi-phase model's single pole with multiplicity $r$ cycles through powers of the same two quantities $1/s$ and $(1-\pi)/\sigma$.

\begin{remark}[What the DRP saves]\label{rem:mp-saving}
In the CCP, the DRP eliminates one new harmonic statistic at each level; here, it eliminates a routine rational computation.  The mechanism---$\alpha_N(1) = \alpha_D(1)$ forcing cancellation of $g^{(k)}$ in $\Delta_k$---is identical.  The DRP is a property of the catastrophe mechanism, not of the base process.
\end{remark}


\subsection{Phase Transition and Convergence Rate}\label{sec:mp-transition}

Specializing Definition~\ref{def:alpha-beta}:
\begin{equation}\label{eq:mp-alpha}
  \alpha = \frac{\rho^r(s\sigma - qr)}{\sigma(1-\rho^r)}, \qquad |\alpha| \sim \frac{qr}{\sigma}\,\rho^r, \quad p = \rho^r \qquad (r \to \infty).
\end{equation}
Since $q$ is fixed, Corollary~\ref{cor:exp-limit} applies directly: $T/E[T] \xrightarrow{d} \Exp(1)$ as $r \to \infty$.  Without catastrophe, $T_0 \sim \mathrm{NegBin}(r, \pi)$ satisfies a CLT: $(T_0 - r/\pi)/\sqrt{r(1-\pi)/\pi^2} \xrightarrow{d} N(0,1)$.  A single catastrophe mechanism thus induces a \textbf{Gaussian-to-exponential phase transition}---distinct from the CCP's Gumbel-to-Exponential transition, but driven by the same algebraic cause: the coefficient equality $\alpha_N(1) = \alpha_D(1)$ forces the $\theta^1$-coefficient in the Laplace expansion to vanish exactly (Proposition~\ref{prop:alpha-properties}\,\eqref{lem:identity}).

The sharp two-sided bound (Theorem~\ref{thm:two-sided}) gives $\dK(T/E[T],\, \Exp(1)) \asymp p + |\alpha| \asymp r\rho^r$ as $r \to \infty$: convergence is \emph{exponentially fast} in $r$ with a linear prefactor, contrasting with the CCP's polynomial rate $\Theta(m!\,m\ln n/n^m)$.

\begin{center}
\small
\renewcommand{\arraystretch}{1.2}
\begin{tabular}{lll}
\toprule
& \textbf{CCP ($n \to \infty$, $m$ fixed)} & \textbf{Multi-phase ($r \to \infty$, $\pi, q$ fixed)} \\
\midrule
Without catastrophe & Gumbel & Gaussian \\
With catastrophe & Exponential & Exponential \\
Rate template & $p + |\alpha| \asymp m!\,m\ln n/n^m$ & $p + |\alpha| \asymp r\rho^r$ \\
Rate type & polynomial in $n$ & exponential in $r$ \\
\bottomrule
\end{tabular}
\end{center}

\begin{remark}[CV criterion and forced catastrophe]\label{rem:cv}
The coefficient of variation of $T_0 \sim \mathrm{NegBin}(r, \pi)$ satisfies $\CV^2 = (1-\pi)/r < 1$ for $r \geq 2$.  By the Pal--Reuveni criterion (\cite{Reuveni16}), adding voluntary restart to this base process is always detrimental.  The present application demonstrates that the DRP framework remains fully operative---and produces non-trivial distributional information---even when catastrophe is imposed against the process's interest.  The algebraic structure does not consult the CV criterion.
\end{remark}

\section{Related Work}\label{sec:related}

\paragraph{Stochastic Resetting and the PGF Formula.}
The renewal framework for stochastic resetting---developed in~\cite{EM11, Reuveni16, PalReuveni17, CS18} and surveyed in~\cite{EMS20}---provides the attempt decomposition underlying Theorem~\ref{thm:pgf}.  The PGF formula was derived independently by~\cite{FP21} and~\cite{BP21}.  These works primarily use the resulting transform for computation or restart optimization, without identifying the affine structure or establishing convergence rates. Our framework places the CV-unity criterion of~\cite{Reuveni16}
in a broader structural context: the dependence on first and second moments appears as the $k=2$ shadow of the derivative reduction principle. 
Non-memoryless protocols~\cite{EuleMetzger16, CY22, KBHR25} fall outside this affine transform class; our characterization theorems (Theorems~\ref{thm:char},~\ref{thm:ct-char}) identify the precise algebraic boundary.

\paragraph{Linear-Fractional Structure across Domains.}
Geometric compounding inherently produces linear-fractional PGFs, exploited in branching process theory by~\cite{Agresti74, Sagitov13} (see also~\cite{GrosjeanHuillet17, LindoSagitov18}), in catastrophe queueing by~\cite{BGR82, Gelenbe91} (see~\cite{KKA00, EF03, BW17} for extensions), and observed empirically in algorithmic restart, where~\cite{GSCK00} noted that restart converts heavy-tailed SAT runtimes into geometric-like distributions (see~\cite{LSZ93, vMW04, Wolter10}).  All these works treat linear-fractional structure as a convenient tractable family or an empirical regularity.  Our characterization theorem reveals the complementary perspective: this structure is \emph{diagnostic}---it identifies the geometric-tail class among all age-based catastrophe mechanisms---and the M\"obius-rigidity argument applies it as a characterization tool rather than a model assumption.

\paragraph{Exponential Approximation of Geometric Sums.}
\cite{Renyi56} proved that normalized geometric sums converge to $\Exp(1)$.  \cite{Brown90} obtained optimal-order Kolmogorov bounds $d_K \leq C_1 p$; \cite{PR11} introduced Stein's method for this setting, achieving $O(p)$ Wasserstein and $O(\sqrt{p})$ Kolmogorov bounds (see~\cite{Brown15, PRR13} for refinements, and~\cite{GK96, Kalashnikov97} for comprehensive treatments).  Our Bridge Theorem matches Brown's $O(p)$ rate for the geometric-sum component; the sharp two-sided bound $\dK \asymp p + |\alpha|$ (Theorem~\ref{thm:two-sided}) closes a logarithmic gap from the smoothing inequality by exploiting the specific two-parameter structure produced by the affine PGF analysis.  Stein's method does not directly yield this two-parameter template: existing Stein bounds apply to the pure geometric sum~$U$ (achieving $O(\sqrt{p})$ in Kolmogorov distance~\cite{PR11}, weaker than Brown's $O(p)$), while the $|\alpha|$-term arises from the successful-attempt perturbation~$B$---an additive component outside the geometric-sum framework that the Bridge Theorem handles via direct probabilistic arguments (Lemmas~\ref{lem:scaling}--\ref{lem:additive}).  The matching lower bound, obtained by evaluating the Laplace expansion at a fixed~$\theta$, is likewise orthogonal to the Stein approach.

\cite{CostacequeDecreusefond26} develop Stein's method for the Gumbel approximation and obtain a $(\log n/n)$ rate for the classical coupon collector in a smooth Wasserstein-$1$ metric.

\paragraph{Coupon Collector Variants and Phase Transitions.}
The classical Gumbel limit was established by~\cite{ER61} (see~\cite{BB65, FGT92, FS09, DP12, Neal08} for refinements and extensions).  The closest prior phase-transition result is~\cite{NK06}, who identified a Gumbel-to-Gaussian transition in the CCP with bonuses.  \cite{JT24} study a closely related reset-button CCP, deriving the waiting-time distribution in the unequal-probability setting and expected-value formulas with asymptotics in the equal-probability case.  Our contribution is the structural theory (affine PGF framework, DRP, sharp exponential approximation), and the identification of the Gumbel-to-Exponential phase transition---which is qualitatively different from the Gumbel-to-Gaussian transition of~\cite{NK06} (it reflects memoryless catastrophe, not a CLT effect).

Our phase transitions differ from those of~\cite{PP19}, who study transitions in the \emph{optimal restart rate}.  Ours are in the \emph{universality class} of the limit distribution---Gumbel $\to$ Exponential, Gaussian $\to$ Exponential---driven by the coefficient equality $\alpha_N(1) = \alpha_D(1)$ killing the first-order Laplace error.

\section{Further Directions}\label{sec:further}

\paragraph{Stability and Refinement of Affine Structure.}
Theorem~\ref{thm:char} characterizes affine PGF structure as equivalent to geometric-tail catastrophe, with full memorylessness singled out by $\lambda = 1$ (Corollary~\ref{cor:full-memoryless}).  The gap---exactly one degree of freedom, the first-step hazard $h_1$---closes in continuous time (Theorem~\ref{thm:ct-char}).  In the discrete setting, supplementary algebraic conditions (e.g., requiring the successful-attempt PGF $G_B$ to also depend on $g$ at a single shifted point) may close this gap; formalizing this would strengthen the characterization.

A related stability question: for near-memoryless mechanisms with $|h_t - q| \leq \varepsilon$ for $t \geq 2$, does the DRP cancellation error $|\Delta_k - (-ks^{k-1}D_{k-1})|$ admit an $O(\varepsilon)$ bound?  More broadly, quantifying DRP degradation under Gamma-distributed inter-reset times (shape parameter near~$1$) would connect the algebraic theory to the restart optimization literature.  At the application level, the CCP sharp bound (Theorem~\ref{thm:ccp-sharp}) gives leading constants in $[1,2]$; numerical evidence suggests the exact constant is closer to~$1$, likely requiring a finer analysis of the additive perturbation step.

\paragraph{Partial Reset.}
When catastrophe clears each coupon independently with probability $r \in (0,1)$, the i.i.d.\ attempt structure of~\eqref{eq:decomp} breaks down: partial reset creates state-dependent catastrophe that cannot be described by a hazard sequence $(h_t)_{t \geq 1}$.  Theorem~\ref{thm:char} predicts that affine PGF structure is lost, but a quantitative understanding---what replaces the DRP, what the correct rate template is---remains open.  Partial reset interpolates between the memoryless regime ($r = 1$, exponential limit) and the classical CCP ($r = 0$, Gumbel limit); understanding this interpolation would reveal whether the Gumbel-to-Exponential phase transition is sharp or admits a crossover regime.  A probabilistic proof of the DRP---perhaps via coupling---might extend to this setting where the PGF algebra breaks down.

\paragraph{Non-Uniform Coupon Probabilities.}
For non-uniform CCP with probabilities $p_1, \ldots, p_n$, the PGF no longer telescopes, but the general theory (Theorem~\ref{thm:pgf}, DRP, exponential approximation) applies unchanged to the non-uniform base process.  The challenge is evaluating the CCP-specific quantities ($D_1$, $\alpha$) in terms of the non-uniform parameters and understanding how the phase transition depends on the coupon probability vector.

\paragraph{Universality Classes of Completion-Time Limits.}
The present paper establishes that memoryless catastrophe (i.i.d.\ attempts, affine PGF structure) produces exponential limits.  A structurally contrasting regime may arise in non-renewal recycling mechanisms with weakly dependent increments, where CLT-type arguments yield log-normal completion-time limits instead.  This contrast raises a natural question: \emph{what properties of the catastrophe or recycling mechanism determine the universality class of the completion-time limit law?}

\newpage

\bibliography{cas-refs}

\end{document}